\RequirePackage{fix-cm}
\documentclass{article}

\usepackage{arxiv}
\usepackage{lipsum}
\usepackage{placeins}
\usepackage[italicComments=true,indLines=true,noEnd=false]{algpseudocodex} 
\usepackage{algorithm} 
\usepackage{float}
\usepackage{amsthm}
\usepackage[english]{babel}
\usepackage[T1]{fontenc}
\usepackage[utf8]{inputenc} 
\usepackage{hyperref}
\usepackage{url}
\usepackage{booktabs}
\usepackage{amssymb,amsfonts,amsmath,amsthm,amstext,amsbsy}
\usepackage{nicefrac}
\usepackage[version=4]{mhchem}
\usepackage{microtype}
\usepackage{mathrsfs}
\usepackage{graphicx}
\usepackage[numbers,sort&compress]{natbib}
\usepackage{doi}
\usepackage{xcolor}
\usepackage{enumitem}
\usepackage{caption}
\usepackage[title]{appendix}
\usepackage[capitalise,nameinlink]{cleveref}
\usepackage[bottom]{footmisc}
\usepackage{multirow}
\usepackage{subcaption} 
\usepackage{pgfplots}
\usepackage{tikz}
\usetikzlibrary{positioning, arrows.meta}
\usepackage{array} 

\newtheorem{proposition}{Proposition}[section]

\newtheorem{remark}{Remark}

\usepackage{tabularx}      
\usepackage[table]{xcolor} 

\newcolumntype{C}{>{\centering\arraybackslash}X}
\definecolor{ao(english)}{rgb}{0.0, 0.5, 0.0}

\newcounter{savesecnumdepth}

\pgfplotsset{compat=1.17}

\title{Autoencoders versus Numerical Analysis--Informed Manifold Learning
for Navier--Stokes Flows}

\author{
\textbf{Alessandro Della Pia\textcolor{blue}{$^{1}$}\thanks{Work performed while at Scuola Superiore Meridionale, School for Advanced Studies, Naples 80138, Italy}, 
Lucia Russo\textcolor{blue}{$^{2}$}, 
Ioannis Kevrekidis\textcolor{blue}{$^{3}$}, 
Constantinos Siettos\textcolor{blue}{$^{4}$}\thanks{Corresponding author, email: \texttt{constantinos.siettos@unina.it}}} \\
\\
\textcolor{blue}{$^{(1)}$}
SISSA, International School for Advanced Studies,
\emph{Mathematics Area, mathLab}, Trieste, 34136, Italy
\\
\textcolor{blue}{$^{(2)}$} Department of Chemical and Biomolecular Engineering, Department of Applied Mathematics and \\
\hspace{0.39cm} Statistics \& Department of Urology, \emph{Johns Hopkins University}, Baltimore, MD, 21218, USA
\\
\textcolor{blue}{$^{(3)}$} Institute of Science and Technology for Energy and Sustainable Mobility (STEMS), \\
\emph{National Research Council (CNR)}, Naples, 80125, Italy
\\
\textcolor{blue}{$^{(4)}$} Department of Mathematics and Applications ‘‘Renato Caccioppoli", \\ \emph{University of Naples ‘‘Federico II"}, Naples, 80126, Italy
}

\date{\today}

\begin{document}

\maketitle

\begin{abstract}
Autoencoders (AEs) have become a dominant approach to nonlinear latent-space construction in data-driven reduced-order modelling (ROM), with their decoders lifting latent representations back to the ambient state space. Their prominence, however, has overshadowed an established alternative: manifold-learning methods grounded in classical numerical analysis. We revisit this alternative using Parsimonious Diffusion Maps (PDMs), benchmarking them against Proper Orthogonal Decomposition (POD)-based ROMs and several convolutional AE architectures for the two-dimensional incompressible flow past a rotating cylinder ---a bifurcating Navier-Stokes (NS) system organized by a codimension-2 Bogdanov-Takens point and its associated Hopf, saddle-node, and homoclinic bifurcations. Our approach uses PDMs to identify a parsimonious and interpretable set of intrinsic latent coordinates and to estimate their dimension directly from data. Gaussian process regression then learns the latent dynamics, while convex $K$-nearest-neighbor ($K$-NN) interpolation in PDMs space constructs the pre-image map, for which we establish pointwise consistency. 
The resulting nonlinear ROM substantially outperforms POD-based ROMs and achieves reconstruction and prediction accuracy comparable to ---and, in some bifurcating regimes, better than--- that of AE-based ROMs. At the same time, latent-variable learning with PDMs requires orders of magnitude less computational time than AE training.
\end{abstract}


\section{Introduction}
\label{sec:intro}

Reduced-order modelling has become a standard strategy in computational fluid dynamics (CFD): the resulting ROMs capture the dominant features of complex flows at a small fraction of the cost and memory of full-order simulation~\cite{haller2025modeling,temam1995navier,titi1990approximate,quarteroni2007numerical,quarteroni2014reduced,vinuesa_brunton_ML-CFD,hijazi2020data,Pichi}. By replacing high-fidelity simulations with low-dimensional surrogates, ROMs enable rapid predictions and facilitate tasks such as parametric studies, bifurcation and stability analysis, flow control, and design optimization, which would otherwise be prohibitively expensive in many cases of practical interest.

The foundational idea behind ROMs is that, despite the high dimensionality and nonlinearity of fluid flows, their dynamics often evolve on low-dimensional latent structures that can be parameterized by a suitable set of coordinates. A classical such example is the derivation of the celebrated Lorenz system by paper and pencil as a low-dimensional approximation of the Navier--Stokes--Boussinesq equations describing Rayleigh--B\'enard convection~\cite{lorenz1963deterministic}. In general, the construction of ROMs can be interpreted as a three-stage process consisting of: (i) the encoding stage, i.e., the discovery of a suitable low-dimensional (latent) coordinate system that encodes the essential/emergent flow dynamics, (ii) the modelling of the temporal evolution of these latent low-dimensional coordinates via reduced-order surrogate models, and (iii) the reconstruction (lifting/decoding) of the high-dimensional flow field from the reduced representation. The effectiveness of this ``embed--learn--lift'' framework, rooted in the celebrated Equation-Free (EF) multiscale methodology~\cite{kevrekidis2003equation,RUSSO200751} introduced in the early 2000s, critically depends on the accurate identification of the latent space.

Traditional approaches to ROM construction in fluid dynamics are typically intrusive, as they rely on explicit knowledge of the governing equations. In such methods, the Navier--Stokes equations are projected onto a reduced basis, most commonly obtained through Proper Orthogonal Decomposition (POD)~\cite{deane1991low,ma2002low,lassila2014model,ballarin2016pod,stabile2018finite,rowley2005model,brunton2020machine}. This procedure leads to reduced dynamical systems that approximate the original flow behavior. Related developments arise in the context of dynamical systems theory, where simplified models are constructed near bifurcation points by systematically reducing the governing equations to their essential components. Examples include normal form theory~\cite{seydel2009practical}, invariant and center manifold approaches~\cite{temam1995navier,hirsch1970invariant,gallay2002invariant,siettos2014equation,wiggins2013normally,siettos2022numerical}, spectral submanifold theory~\cite{haller2016nonlinear,breunung2018explicit,buza2024spectral}, approximate inertial manifolds~\cite{temam1995navier,titi1990approximate,Colombo2025ROMHopf}, and quantized local reduced-order models in time~\cite{COLANERA2025118393}. These approaches provide valuable insight into the onset of instabilities and the structure of solution branches in bifurcation scenarios, while significantly reducing computational complexity. However, their reliance on the governing equations limits their applicability in fully data-driven settings. On the other hand, non-intrusive ROMs operate solely on data and do not require explicit knowledge of the underlying equations. The exponential growth of machine learning (ML) and data-driven techniques has greatly expanded the capabilities of such approaches~\cite{deane1991low,hijazi2020data,stabile2018finite,Hesthaven_2018,girfoglio2022pod}. In this framework, the identification of reduced coordinates is commonly formulated as a manifold learning problem, where the goal is to uncover a low-dimensional structure embedded in high-dimensional flow data. Classical linear techniques such as POD remain widely used because of their simplicity and the availability of explicit reconstruction mappings. However, their linear nature limits their ability to represent parsimoniously strongly nonlinear flow dynamics including multistability and turbulence \cite{DellaPia_diffusion_2024}.

To address this limitation, nonlinear manifold learning methods (see also in \cite{gallos2021construction,papaioannou2022time} for a review) such as Kernel PCA \cite{scholkopf1997kernel}, Diffusion Maps (DMs)~\cite{coifman2005geometric,coifman2006geometric,nadler2006diffusion,dsilva2018parsimonious,galaris2022numerical,Kevrekidis_DM,Patsatzis_2023,DellaPia_diffusion_2024}, ISOMAP~\cite{tenenbaum2000global}, Local Linear Embedding~\cite{roweis2000nonlinear}, and Laplacian Eigenmaps~\cite{belkin2003laplacian,belkin2008towards} have been introduced to better capture the intrinsic geometry of complex datasets. These methods can identify nonlinear low-dimensional embeddings that more faithfully represent the underlying dynamics, including in regimes characterized by complex or turbulent/chaotic behavior~\cite{DellaPia_diffusion_2024}. A major challenge in this context is the so-called pre-image problem, i.e., reconstructing the high-dimensional state from the latent representation. Unlike linear methods, where reconstruction reduces to a straightforward least-squares problem with closed-form solutions, nonlinear embeddings require approximate reconstruction techniques such as geometric harmonics, kernel methods, local interpolation schemes including $K$-nearest-neighbors~\cite{chiavazzo2014reduced,papaioannou2022time,chin2024enabling,Patsatzis_2023}, or the recently introduced random-feature/multi-scale neural decoders ensuring mass preservation (RANDSMAPs \cite{patsatzis2026}).

More recently, alongside these developments, autoencoders have attracted significant attention and undergone rapid development as nonlinear dimensionality-reduction tools ~\cite{chen2018molecular,gonzalez2018deep,fukami2020convolutional,li2020scalable,vlachas2022multiscale,floryan2022data,oommen2022learning,Eivazi,koronaki2024nonlinear,kontolati2024learning,zeng2024autoencoders,simpson2024vprom}. Originally developed in the context of autoassociative neural networks and nonlinear principal component analysis \cite{kramer1991nonlinear}, they provide, usually via deep neural networks, a nonlinear parametrization of high-dimensional data through an encoder that maps snapshots to latent variables and a decoder that reconstructs the corresponding ambient-space states. Autoencoder-based methods have been successfully applied to a wide range of fluid-flow problems, often in conjunction with techniques such as sparse identification of nonlinear dynamical systems (SINDy)~\cite{Brunton_SINDy,Hasegawa_2020,Loiseau_2020,Callaham_2022,Champion_SINDy,CONTI2023116072}, neural operators~\cite{kontolati2024learning}, Gaussian Process Regression~\cite{wan2017reduced,stephenson2018accelerating,ma2021data,ortali2022gaussian}, and neural-network-based sequence modelling~\cite{bertalan2019learning,lee2020coarse,arbabi2020linking,lee2023learning,dietrich2023learning,fabiani2024task,Srinivasan_2019}. 
Despite their flexibility and strong empirical performance, autoencoder-based methods present several important challenges. Their construction requires the selection of many design choices and hyperparameters, including the encoder--decoder architecture, the number and type of hidden layers, the number of neurons or filters, activation functions, regularization strategies, optimization parameters, and the dimension of the latent space. These choices are often made through computationally intensive trial-and-error procedures and may strongly affect the resulting ROMs. Moreover, training autoencoders amounts to solving a highly nonconvex optimization problem. From a theoretical standpoint, the global training of neural networks is known to be computationally intractable in worst-case settings, with classical results showing NP-hardness~\cite{froese2023training} even for very small architectures. Thus, although practical algorithms such as stochastic gradient descent can be highly effective, they do not generally provide guarantees of reaching a globally optimal latent representation. Although this does not imply that every practical autoencoder training instance is intractable, it highlights that autoencoder training lacks general global optimality guarantees and may require substantial empirical tuning.

In this work we revisit the increasingly common reliance on autoencoder-based latent spaces in data-driven ROMs, comparing them against manifold-learning methods grounded in classical numerical analysis, which identify intrinsic coordinates directly from data and yield an interpretable latent representation. We approach this comparison through a data-driven ROM framework built on Parsimonious Diffusion Maps (PDMs), a nonlinear manifold-learning technique designed to extract low-dimensional, latent coordinates. Diffusion Maps and their parsimonious variant were recently introduced to the fluid dynamics community for the development of physics-infused machine-learning models of one-dimensional thin-film flows~\cite{Kevrekidis_DM} and for two-dimensional flows with turbulence~\cite{DellaPia_diffusion_2024}. They were subsequently used to construct ``normal-form'' ROMs, namely reduced models of minimal dimensionality, for bifurcation analysis in latent space in flow configurations exhibiting codimension-1 bifurcations of stationary states, such as pitchfork and Hopf bifurcations, as well as limit-cycle bifurcations~\cite{dellapia2026surrogate}.
Here, we provide, to the best of our knowledge, the first extensive comparison between autoencoder-based ROMs and nonlinear manifold learning techniques, and in particular Parsimonious-Diffusion-Maps-based ROMs in a complex two-dimensional Navier--Stokes flow; a comparison with the classical (linear manifold learning) POD is also performed. The selected benchmark flow configuration features rich emergent dynamics, including multiple flow regimes and bifurcation-induced transitions, and therefore provides a stringent test case for assessing the accuracy, interpretability, computational cost, and latent-space structure of the three approaches. We note that autoencoder-based architectures have recently been compared with POD in the context of intraventricular flows~\cite{lazpita2025critical}. While this comparison provides useful insights, it should be interpreted in light of the different nature of the two approaches: POD seeks an optimal linear subspace representation, whereas autoencoders aim to learn a nonlinear parametrization of the data manifold. Therefore, assessing the role of autoencoders in nonlinear ROMs calls for comparisons not only with linear projection methods, but also with nonlinear manifold-learning techniques such as Diffusion Maps, whose latent coordinates admit an interpretation grounded in numerical analysis.
For both PDM- and autoencoder-based ROMs, the latent dynamics are learned through Gaussian Process Regression (GPR), yielding accurate non-intrusive surrogate models with uncertainty quantification. For the PDM framework, reconstruction of the full-order state is achieved by using $K$-nearest-neighbors ($K$-NN) interpolation~\cite{chin2024enabling} to approximate the pre-image, or lifting, map from diffusion coordinates to the high-dimensional flow fields. The resulting ROM is then systematically compared with classical linear POD embeddings and nonlinear autoencoder-based ROMs, demonstrating that PDMs offer a balanced and efficient alternative for data-driven reduced-order modelling of complex fluid flows. In particular, the performance of the three approaches is assessed via the benchmark problem of two-dimensional flow past a rotating cylinder, governed by the Reynolds number $Re$ and the rotation rate $\alpha$. In this setting and in the parameter space considered, the NS equations exhibit a rich sequence of dynamical transitions associated with a codimension-2 Bogdanov--Takens bifurcation, involving the interaction of Andronov--Hopf, saddle-node, and homoclinic bifurcation curves~\cite{Sierra}. Such multi-regime behavior provides a stringent benchmark for reduced-order modelling approaches.

The remainder of the paper is organized as follows. Section~\ref{sec:methodology} presents the data-driven ROM formulation within the three-stage embed--learn--lift framework (Sections~\ref{subsec:manifold}-~\ref{subsec:decoder}), with additional details on GPR reported in Appendix~\ref{app:GPR}. The methodology then focuses on Parsimonious Diffusion Maps in Sections~\ref{subsec:meth_PDMs} and \ref{sec:kNN}, while autoencoders and POD, which are used for systematic comparison, are briefly reviewed in Sections~\ref{subsec:AE} and~\ref{subsec:POD}, respectively. The benchmark flow configuration is described in Section~\ref{sec:flow_config}, while details of the data generation procedure are reported in Appendix~\ref{app:NS_eq}. The results are presented and discussed in Section~\ref{sec:results} and Appendix~\ref{app:UQ_AE}, and concluding remarks are finally drawn in Section~\ref{sec:conclusions}.


\section{Methodology}
\label{sec:methodology}
Data-driven reduced-order modelling (ROM) of fluid flows (and, more generally, of complex dynamical systems governed by partial differential equations) can be formulated within a three-stage framework, which is schematized in Fig.~\ref{graphics_abs}. Given a high-dimensional state, the objective is to construct a low-dimensional surrogate model that accurately captures the underlying dynamics while enabling efficient analysis and prediction.
\subsection{Nonlinear Manifold Learning}
\label{subsec:manifold}
The first stage consists of manifold learning, namely identifying a mapping from the original high-dimensional state space to a low-dimensional latent space. Given a point cloud 
\begin{equation}
\mathcal{X}_M = \{\mathbf{x}_m\}_{m=1}^{M}\subset\mathbb{R}^{N}, 
\end{equation}
whose elements are assumed to lie on, or near, a smooth low-dimensional manifold $\mathcal{M}\subset\mathbb{R}^{N}$, manifold learning aims to construct finite-dimensional operators, such as graph Laplacians or kernel matrices, whose spectral decomposition reveals intrinsic coordinates of the data. Depending on the operator, the resulting data-driven coordinate map 
\begin{equation}
\Phi_M : \mathcal{X}_M \to \mathbb{R}^{d}, \qquad \mathbf{y}_m = \Phi_M(\mathbf{x}_m), 
\label{eq:HD2LDmap}
\end{equation}
approximately preserves certain geometric quantities, such as geodesic distances in ISOMAP, diffusion distances in Diffusion Maps, or local neighborhood relationships in Laplacian Eigenmaps. The retained dimension \(d\) is usually selected according to spectral decay, spectral gaps, or parsimonious-coordinate criteria, and need not coincide with the intrinsic dimension of the underlying manifold. When the point cloud provides a sufficiently dense discretization of a compact region \(\mathcal{W}\subset\mathcal{M}\), and when \(\mathcal{M}\) is endowed with the Riemannian metric induced by the ambient Euclidean space, graph-based operators can be related, under suitable assumptions, to continuum operators on \(\mathcal{M}\). These assumptions typically concern the sampling distribution, kernel bandwidth, normalization of the empirical operator, and spectral separation of the retained modes. For graph Laplacians associated with Laplacian Eigenmaps or Diffusion Maps, the limiting operator is typically a Laplace--Beltrami-type operator, possibly modified by the sampling density \cite{belkin2003laplacian,belkin2006convergence,belkin2008towards,coifman2005geometric,nadler2006diffusion}. In this sense, the empirical eigenvectors provide discrete approximations of continuum eigenfunctions, and the data-driven coordinate map $\Phi_M$ can be seen as a discrete approximation of the continuum map
\begin{equation}
\Phi : \mathcal{W} \to \mathbb{R}^{d}. 
\end{equation}
\paragraph{Out-of-sample extension.} Given a new point $\mathbf{x}^{*}\notin \mathcal{X}_M$, but assumed to lie on or near $\mathcal{M}$, the out-of-sample extension problem consists of computing its latent representation $\Phi_M(\mathbf{x}^{*})$. A standard approach is the Nystr\"om extension, which evaluates the previously computed eigenfunctions at the new point. In component form, and up to the normalization used in the particular kernel or graph operator, this can be written schematically as \cite{coifman2006geometric,papaioannou2022time} 
\begin{equation}
\phi_{\ell,M}(\mathbf{x}^{*}) = \frac{1}{\lambda_\ell} \sum_{m=1}^{M} k_M(\mathbf{x}^{*},\mathbf{x}_m)\,u_{\ell,m}, \qquad \ell=1,\ldots,d, 
\end{equation}
where $(\lambda_\ell,\mathbf{u}_\ell)$ denotes the $\ell$-th eigenpair of the normalized kernel or graph operator, $u_{\ell,m}$ is the $m$-th component of the corresponding eigenvector, and $k_M$ is the associated normalized kernel. The out-of-sample coordinate map reads:
\begin{equation}
\Phi_M(\mathbf{x}^{*}) = \big( \phi_{1,M}(\mathbf{x}^{*}),\ldots,\phi_{d,M}(\mathbf{x}^{*}) \big). 
\end{equation}
With appropriate normalization and sampling assumptions, this provides a consistent approximation of the continuum coordinate map $\Phi(\mathbf{x}^{*})$.

\subsection{Learning Surrogate Models on Latent Spaces}
\label{subsec:learning}
The second stage consists of learning the governing dynamics in the latent space. This is achieved through a regression model of the form
\begin{equation}
\mathbf{y}(t+T_p) = \boldsymbol{\phi}\bigl(\mathbf{r}(t), \boldsymbol{\chi}\bigr) + \mathbf{e}(t),
\label{eq:regression_basic_map}
\end{equation}
where $T_p$ is the prediction time horizon, $\mathbf{r}(t)$ is a regressor constructed from the latent variables, $\boldsymbol{\chi}$ collects governing parameters or external inputs to the model, and $\mathbf{e}(t)$ represents modelling errors or noise. In contrast to continuous-time formulations, the discrete-time representation avoids the need to estimate time derivatives from data, which can be particularly challenging in the presence of noise or sparsely sampled measurements, as is often the case in experimental datasets.

In the present work, the latent dynamics are learned using Gaussian Process Regression (GPR), which additionally provides a probabilistic estimate of the model uncertainty (see Appendix~\ref{app:GPR}). Note that alternative approaches could also be employed at this stage, such as sparse identification of nonlinear dynamics (SINDy~\cite{Brunton_SINDy}) or artificial neural networks (ANNs), depending on the desired balance between interpretability, flexibility, and computational cost.

\begin{remark}
Here we note that we learn the discrete-time evolution map between successive snapshots of the data, rather than the continuous-time vector field or its time derivatives. Learning time derivatives from real data can be restrictive and, in many applications, unrealistic. It typically requires sufficiently dense temporal sampling, low-noise observations, and a meaningful numerical smoothing/differentiation procedure. These requirements may not be satisfied in experimental, observational, or coarse-sampled data sets. On the other hand, we note that bifurcations detected in the learned discrete map are bifurcations of the sampled-time dynamics and need not coincide, in a one-to-one manner, with the bifurcations of the underlying continuous-time system. For example, an Andronov--Hopf bifurcation of a continuous-time system corresponds, at the level of a discrete map, to a Neimark--Sacker bifurcation. Therefore, unless the learned map is known to approximate the exact flow map of an underlying differential equation, the bifurcation structure should be interpreted primarily as that of the discrete-time sampled system (see the discussion in \cite{rico2000noninvertibility} and more recently in \cite{cui2021pathologies}).
\end{remark}

\subsection{The solution of the pre-image problem}
\label{subsec:decoder}
The third stage consists of solving the so-called pre-image (or lifting) problem in manifold learning, i.e.\ reconstructing the high-dimensional state from a given latent representation. 

Given a latent coordinate $\mathbf{y}\in\mathbb{R}^{d}$, not equal to $\Phi_M(\mathbf{x}_m)$ for any training snapshot, the pre-image problem consists of finding a high-dimensional state $\mathbf{x}\in\mathcal{M}$, such that 
\begin{equation}
\Phi_M(\mathbf{x}) \approx \mathbf{y}. 
\end{equation}
Hence, in the pre-image problem, the latent coordinates are given and one seeks a corresponding high-dimensional realization. The problem is generally ill-posed, because the latent point $\mathbf{y}$ may not lie exactly in the range of the empirical coordinate map, and because the coordinate map is not guaranteed to be globally injective. A generic regularized formulation of the pre-image problem is 
\begin{equation}
\hat{\mathbf{x}} = \operatorname*{arg\,min}_{\mathbf{x}\in\mathbb{R}^{N}} \left\| \Phi_M(\mathbf{x})-\mathbf{y} \right\|_2^{2} + \lambda R(\mathbf{x}), 
\end{equation}
where $R(\mathbf{x})$ is a regularization term penalizing deviations from the data manifold or enforcing additional physical constraints. Note that evaluating $\Phi_M(\mathbf{x})$ for an arbitrary $\mathbf{x}\in\mathbb{R}^{N}$ already requires an out-of-sample extension, since $\mathbf{x}$ is not necessarily one of the training points. In practice, solving this optimization problem directly is often difficult. Therefore, approximate lifting maps 
\begin{equation}
\Gamma_M:\mathbb{R}^{d}\to\mathbb{R}^{N} 
\label{eq:psi}
\end{equation}
are commonly constructed using local interpolation, $K$-nearest-neighbors reconstructions, geometric harmonics, double Diffusion Maps \cite{evangelou2022double,Patsatzis_2023}, and kernel regression \cite{chiavazzo2014reduced}. Unlike out-of-sample extension, the pre-image problem requires additional assumptions or regularization and cannot be resolved by the forward coordinate map alone.

Here, we compare three different approaches: Parsimonious Diffusion Maps (PDMs), convolutional neural network-based autoencoders (AEs), and Proper Orthogonal Decomposition (POD). While these methods differ in how the mapping $\Phi_M$ (and, consequently, $\Gamma_M$) is constructed, the second-stage regression model is kept consistent across all approaches, enabling a systematic comparison of their ability to capture the intrinsic low-dimensional structure of the flow. As detailed in the remainder of this section, different lifting strategies are adopted depending on the method. For Parsimonious Diffusion Maps, the lifting is performed using a $K$-nearest-neighbors ($K$-NN) interpolation in the latent space. For autoencoders, the inverse mapping is directly learned by the decoder network, which provides a nonlinear, data-driven approximation of the lifting operator. For POD, the pre-image problem admits a closed-form solution through the linear reconstruction provided by the modal expansion.

In the following, we describe in detail the three approaches used to construct the mapping $\Phi_M$, along with the corresponding lifting strategies to solve the pre-image problem.

\subsection{The Encoder: Parsimonious Diffusion Maps}
\label{subsec:meth_PDMs}

\begin{figure}
    \centering
    \includegraphics[scale=0.5]{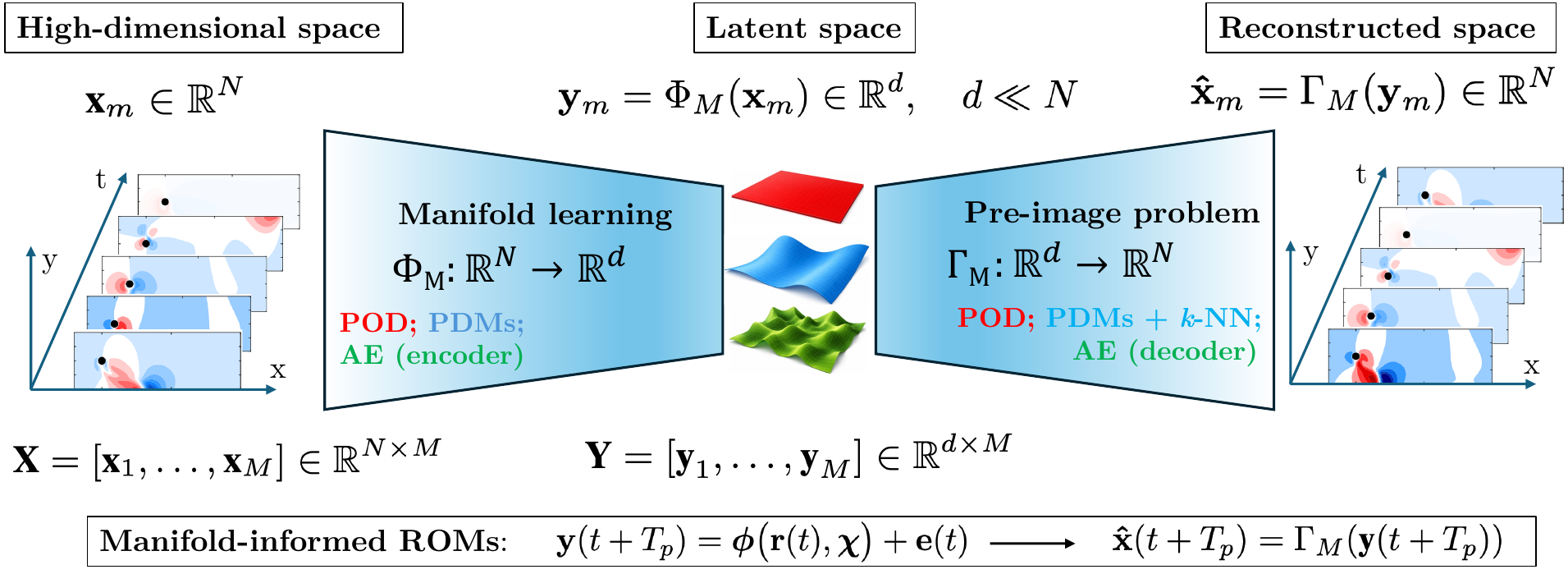}
\caption{Data-driven reduced-order modelling (ROM) in fluid dynamics as a three-stage framework:
(i) projection of high-dimensional Navier--Stokes flow fields onto a latent space via manifold learning (Eq.~\eqref{eq:HD2LDmap});
(ii) learning manifold-informed ROMs by approximating the solution operator in the latent space (Eq.~\eqref{eq:regression_basic_map});
(iii) solving the pre-image problem, i.e., constructing a mapping from the learned latent-space dynamics to the full spatio-temporal dynamics (Eq.~\eqref{eq:psi}), thereby recovering the solution operator of the Navier--Stokes equations.
\label{graphics_abs}}
\end{figure}

Given a high-dimensional state $\mathbf{x} \in \mathbb{R}^N$, Diffusion Maps aim to obtain a low-dimensional embedding $\mathbf{y} \in \mathbb{R}^d$, with $d \ll N$, such that Euclidean distances between points in the embedded space approximate diffusion distances between the original points \cite{nadler2006diffusion}. Following the theoretical formulation and numerical implementation presented in earlier works \cite{dsilva2018parsimonious,holiday2019manifold,Patsatzis_2023,gallos2024data,DellaPia_diffusion_2024}, the algorithm begins by defining a similarity measure between pairs of data points $\mathbf{x}_i, \mathbf{x}_j \in \mathbb{R}^N$, for all $i,j=1,\ldots,M$, where $M$ denotes the total number of snapshots. Using the Euclidean distance
\begin{equation}
d_{ij} = \|\mathbf{x}_i - \mathbf{x}_j\|_2,
\end{equation}
we construct a Gaussian kernel
\begin{equation}
k(\mathbf{x}_i,\mathbf{x}_j) = \exp\left(-\frac{\|\mathbf{x}_i-\mathbf{x}_j\|_2^2}{\epsilon^2}\right),
\end{equation}
which defines the affinity matrix
\begin{equation}
\mathbf{F} = [f_{ij}] = [k(\mathbf{x}_i,\mathbf{x}_j)].
\label{eq:affinity_matrix}
\end{equation}
Here, $\epsilon$ controls the local neighborhood size in the high-dimensional space. In our implementation, we set $\epsilon = \mathrm{median}(d_{ij})$, which promotes a relatively large neighborhood. Other strategies for selecting $\epsilon$ exist \cite{singer2009detecting,gallos2021construction}.

Next, the $M \times M$ Markov transition matrix $\mathbf{M}$ is formed by row-normalizing the affinity matrix:
\begin{equation}
\mathbf{M} = \mathbf{D}^{-1}\mathbf{F},
\qquad
\mathbf{D} = \operatorname{diag}\left(\sum_{j=1}^{M} f_{ij}\right).
\label{eq:Markovian_matrix}
\end{equation}
Each entry $\mu_{ij}$ of $\mathbf{M}$ represents the probability of moving from point $i$ to point $j$ in the high-dimensional space:
\begin{equation}
\mu_{ij} = \operatorname{Prob}\left(X_{t+1}=\mathbf{x}_j \mid X_t=\mathbf{x}_i\right).
\end{equation}
Equivalently, using the kernel notation,
\begin{equation}
\mu_{ij} = \frac{k(\mathbf{x}_i,\mathbf{x}_j)}{\operatorname{deg}(\mathbf{x}_i)},
\qquad
\operatorname{deg}(\mathbf{x}_i)=\sum_{j=1}^{M} k(\mathbf{x}_i,\mathbf{x}_j),
\end{equation}
which recovers Eq.~\eqref{eq:Markovian_matrix}.

The transition matrix $\mathbf{M}$ is similar to the symmetric positive-definite matrix
\begin{equation}
\hat{\mathbf{M}}=\mathbf{D}^{-1/2}\mathbf{F}\mathbf{D}^{-1/2},
\end{equation}
which admits a standard eigendecomposition. In particular, the eigenpairs of $\mathbf{M}$ are written as
\begin{equation}
\mathbf{M}\mathbf{u}_i=\lambda_i \mathbf{u}_i,
\qquad i=1,\ldots,M,
\label{eq:Markov}
\end{equation}
where $\lambda_i \in \mathbb{R}$ are the eigenvalues and $\mathbf{u}_i \in \mathbb{R}^M$ are the corresponding right eigenvectors. The leading non-trivial eigenvectors provide coordinates for the low-dimensional embedding.

The standard Diffusion Maps embedding maps each snapshot $\mathbf{x}_m$ to
\begin{equation}
\mathbf{y}_m = (\lambda_1 u_{1,m}, \ldots, \lambda_d u_{d,m})
\equiv (y_{1,m},\ldots,y_{d,m}),
\qquad m=1,\ldots,M,
\end{equation}
where $u_{i,m}$ denotes the $m$-th component of the $i$-th right eigenvector corresponding to the $i$-th largest non-trivial eigenvalue $\lambda_i$. This embedding approximates the diffusion distance in the high-dimensional space by the Euclidean distance in the embedded space:
\begin{equation}
D_t^2(\mathbf{x}_i,\mathbf{x}_j)
=
\left\|
\mu_t(\mathbf{x}_i,\cdot)-\mu_t(\mathbf{x}_j,\cdot)
\right\|_{L_2,1/\operatorname{deg}}^2
=
\sum_{k=1}^{M}
\frac{
\left(
\mu_t(\mathbf{x}_i,\mathbf{x}_k)-\mu_t(\mathbf{x}_j,\mathbf{x}_k)
\right)^2
}{
\operatorname{deg}(\mathbf{x}_k)
}.
\end{equation}
Here, $\mu_t(\mathbf{x}_i,\cdot)$ denotes the $i$-th row of $\mathbf{M}^t$. In our computations, we use $t=1$.

\begin{remark}
Here, for this encoding step, we note that asymptotically, as the number of
points sampled uniformly from a low-dimensional manifold tends to infinity,
the operator $\frac{1}{\sigma}(\boldsymbol{I}-\boldsymbol{M})$ --- with
$\sigma=\epsilon^{2}$ the kernel scale --- converges to the
Laplace--Beltrami operator of the underlying Riemannian
manifold~\cite{coifman2005geometric,nadler2006diffusion}. The eigenvectors
associated with the largest eigenvalues of $\boldsymbol{M}$ may therefore be
regarded as discrete approximations of the leading eigenfunctions of the
Laplace--Beltrami operator. Since these eigenfunctions provide a faithful
embedding of the manifold~\cite{jones2008manifold}, this justifies using the
eigenvectors of $\boldsymbol{M}$ for practical, data-driven embedding. This
guarantees that the leading Laplace--Beltrami eigenfunctions---approximated
in the discrete setting by the diffusion-maps eigenvectors of
$\mathbf{M}$-provide, locally, a
faithful (low-distortion, injective) coordinate system for the slow
manifold. It implies that states that are distinct on $\mathcal{M}$ are therefore mapped to distinct latent points, and nearby states remain nearby, which is the
precise sense in which the DM latent space is \emph{interpretable}: latent
coordinates carry genuine geometric meaning rather than being an opaque
encoding.
\end{remark}

In practice, the embedded dimension $d$ is often determined from the spectral gap of the eigenvalues of the transition matrix $\mathbf{M}$, under the assumption that the first $d$ leading eigenvalues provide a sufficiently accurate approximation of the diffusion distance between all pairs of points \cite{coifman2008diffusion}. However, this is not always the case, since some of the first eigenvectors may be higher harmonics of previous ones and therefore may not describe genuinely new directions along the dataset. To account for this issue, we further employ Parsimonious Diffusion Maps (PDMs \cite{dsilva2018parsimonious,holiday2019manifold,Patsatzis_2023,gallos2024data}) to select only the eigenvectors that provide unique directions along the dataset, thus yielding the most informative $d$-dimensional embedding.

Given the first $i-1$ eigenvectors, $\mathbf{u}_1,\ldots,\mathbf{u}_{i-1}$, we use a local linear regression model to fit the $i$-th eigenvector $\mathbf{u}_i$ against all previous ones. For each element $m=1,\ldots,M$, this reads
\begin{equation}
u_{i,m} \approx \alpha_{i,m} + \boldsymbol{\beta}_{i,m}^\top \mathbf{U}_{i-1,m},
\end{equation}
where $\alpha_{i,m}\in\mathbb{R}$, $\boldsymbol{\beta}_{i,m}\in\mathbb{R}^{i-1}$, and
\begin{equation}
\mathbf{U}_{i-1,m} = [u_{1,m},\ldots,u_{i-1,m}]^\top.
\end{equation}
The parameters $\alpha_{i,m}$ and $\boldsymbol{\beta}_{i,m}$ are obtained as the solution of the weighted least-squares problem
\begin{equation}
(\alpha_{i,m},\boldsymbol{\beta}_{i,m})
=
\underset{\alpha,\boldsymbol{\beta}}{\arg\min}
\sum_{k\neq m}
\exp\left(
-\frac{\|\mathbf{U}_{i-1,m}-\mathbf{U}_{i-1,k}\|_2^2}{\epsilon_{reg}^2}
\right)
\left(
u_{i,k}-(\alpha+\boldsymbol{\beta}^\top \mathbf{U}_{i-1,k})
\right)^2,
\end{equation}
where $\epsilon_{reg}$ is the kernel scale for the regression algorithm.
Following \cite{dsilva2018parsimonious}, we take $\epsilon_{reg}$ as one-third of the
median of the pairwise distances between $\mathbf{U}_{i-1,m}$, which empirically yields good results.
The exponential weights localize the regression around the point $m$ in the space spanned by the previously selected eigenvectors. Since functional dependence between DMs coordinates may be nonlinear globally but approximately linear locally, the weighted least-squares fit provides a local criterion for detecting redundant eigenvectors. Its solution is
\begin{equation}
\begin{bmatrix} \hat{\alpha}_{i,m} \\ \hat{\boldsymbol{\beta}}_{i,m} \end{bmatrix} = \left( Z_m^\top W_m Z_m \right)^{-1} Z_m^\top W_m \mathbf{u}_{i}^{(-m)}, \quad Z_m = \begin{bmatrix} z_1^\top \\ \vdots \\ z_{m-1}^\top \\ z_{m+1}^\top \\ \vdots \\ z_M^\top \end{bmatrix} \in \mathbb{R}^{(M-1)\times i},  z_k = \begin{bmatrix} 1 \\ \mathbf{U}_{i-1,k} \end{bmatrix} \in \mathbb{R}^{i},
\end{equation}
where 
\[ 
W_m = \operatorname{diag} \left( w_{m1},\ldots,w_{m,m-1},w_{m,m+1},\ldots,w_{mM} \right), \quad w_{mk} = \exp\left( -\frac{ \|\mathbf{U}_{i-1,m}-\mathbf{U}_{i-1,k}\|_2^2 }{\epsilon_{reg}^2} \right), \qquad k\neq m, 
\]
provided that $Z_m^\top W_m Z_m$ is not singular. For statistical tests, the covariance matrix of the estimated coefficients can then be approximated in the standard weighted least-squares form.

The normalized leave-one-out cross-validation error is measured through the local linear fitting coefficient
\begin{equation}
r_i
=
\sqrt{
\frac{
\sum_{m=1}^{M}
\left(
u_{i,m}-(\alpha_{i,m}+\boldsymbol{\beta}_{i,m}^\top \mathbf{U}_{i-1,m})
\right)^2
}{
\sum_{m=1}^{M} u_{i,m}^2
}
}.
\label{eq:LLFc}
\end{equation}
With this definition, a small or negligible value of $r_i$ indicates that the $i$-th eigenvector $\mathbf{u}_i$ can be predicted from the previous eigenvectors $\mathbf{u}_1,\mathbf{u}_2,\ldots,\mathbf{u}_{i-1}$ and therefore corresponds to a repeated eigendirection. Only the eigenvectors associated with sufficiently large values of $r_i$ are retained in order to obtain the most parsimonious representation.

The resulting embedding is constructed from the retained eigenpairs $\{\lambda_i,\mathbf{u}_i\}_{i=1}^{d}$. The restriction operator evaluated at a data point $\mathbf{x}_m$ is therefore given by
\begin{equation}
\Phi_M(\mathbf{x}_m)
=
(\lambda_1 u_{1,m},\ldots,\lambda_d u_{d,m})
\equiv
(y_{1,m},\ldots,y_{d,m})
=
\mathbf{y}_m \in \mathbb{R}^{d},
\qquad m=1,\ldots,M.
\label{eq:DMsmap}
\end{equation}

For new, unseen points, we employ the Nystr\"om method \cite{nystrom1929uber,coifman2006geometric,chiavazzo2014reduced,evangelou2022double,Patsatzis_2023}. 

\subsection{The Decoder: The solution of the pre-image problem with K-NN}
\label{sec:kNN}
Here, to solve the pre-image (lifting) problem, we have used a
$K$-nearest-neighbors ($K$-NN) approach combined with local convex
interpolation~\cite{chin2024enabling}. Given a target latent state
$\mathbf{y}$, let $\{\mathbf{y}_{S(j)}\}_{j=1}^{K}$ denote its
$K$ nearest neighbors among the latent representations of the training data,
and let $\{\mathbf{x}_{S(j)}\}_{j=1}^{K}\subset \mathcal{M}$ be the corresponding
high-dimensional snapshots. The reconstructed full-order state is then
approximated as a convex combination of these neighboring snapshots,
\begin{equation}
\hat{\mathbf{x}}
=
\Gamma_M(\mathbf{y})
=
\sum_{j=1}^{K} b_j \mathbf{x}_{S(j)} .
\end{equation}
The interpolation coefficients \(b_j\) are chosen so as to represent
\(\mathbf{y}\) locally in the embedded space and satisfy the convexity
constraints
\begin{equation}
\sum_{j=1}^{K} b_j = 1,
\qquad
b_j \geq 0,
\qquad j=1,\ldots,K .
\end{equation}
Equivalently, the weights may be obtained by solving the constrained local
least-squares problem
\begin{equation}
\{b_j\}_{j=1}^{K}
=
\arg\min_{\{b_j\}_{j=1}^{K}}
\left\|
\mathbf{y}
-
\sum_{j=1}^{K} b_j \mathbf{y}_{S(j)}
\right\|_2^2
\quad
\text{subject to}
\quad
\sum_{j=1}^{K} b_j = 1,
\quad
b_j \geq 0 .
\end{equation}
This defines a local convex approximation of the inverse map from the
latent manifold to the ambient state space and enables reconstruction of the
high-dimensional flow field associated with a predicted latent state. When
made explicit, the dependence of the weights on the query point is denoted
$b_j(\mathbf{y})$, as used in the consistency result below.

We now state and prove the following consistency result.
\begin{proposition}[Pointwise consistency of the convex $K$-NN pre-image map]
\label{prop:knn_preimage_consistency}
Let $\mathcal M\subset\mathbb R^N$ be a $d$-dimensional $C^2$
embedded manifold, let $\mathcal W\subset\mathcal M$ be open, and let
$\mathcal U\subset\mathbb R^d$ be open. Assume that
$(\mathcal W,\Phi)$, with
$\Phi:\mathcal W\to\mathcal U$,
is a $C^2$ coordinate chart on $\mathcal M$ and denote its inverse
parametrization by
\begin{equation}
\Gamma=\Phi^{-1}:
\mathcal U\to\mathcal W\subset\mathcal M\subset\mathbb R^N,
\end{equation}
i.e., $\Gamma$ is a $C^2$ embedding of $\mathcal U$ onto
$\mathcal W$.

Fix $\mathbf y\in\mathcal U$ and $K\in\mathbb N$. For each sample size $M$, let
$\mathbf y_{S_M(1)},\ldots,\mathbf y_{S_M(K)}$ be the latent points
selected by the $K$-NN construction, with corresponding points
$\mathbf x_{S_M(j)}\in\mathcal W$ satisfying
\begin{equation}
\mathbf y_{S_M(j)}
=
\Phi\bigl(\mathbf x_{S_M(j)}\bigr),
\qquad j=1,\ldots,K.
\end{equation}
Let the coefficients $b_{j,M}(\mathbf y)$ satisfy
\begin{equation}
b_{j,M}(\mathbf y)\geq0,
\qquad
\sum_{j=1}^{K}b_{j,M}(\mathbf y)=1,
\end{equation}
and define
\begin{equation}
\Gamma_M(\mathbf y)
=
\sum_{j=1}^{K}
b_{j,M}(\mathbf y)\mathbf x_{S_M(j)}.
\end{equation}
Set
\begin{equation}
q_M(\mathbf y)
=
\left\|
\mathbf y-
\sum_{j=1}^{K}
b_{j,M}(\mathbf y)\mathbf y_{S_M(j)}
\right\|_2
\end{equation}
and
\begin{equation}
h_M(\mathbf y)
=
\max_{1\leq j\leq K}
\left\|
\mathbf y_{S_M(j)}-\mathbf y
\right\|_2.
\end{equation}
Assume that
\begin{equation}
h_M(\mathbf y)\longrightarrow0
\qquad\text{as }M\to\infty.
\end{equation}
Then there exist $C_{\mathbf y}>0$ and $M_0\in\mathbb N$,
independent of $M$, such that
\begin{equation}
\left\|
\Gamma_M(\mathbf y)-\Gamma(\mathbf y)
\right\|_2
\leq
C_{\mathbf y}
\left(
q_M(\mathbf y)+h_M(\mathbf y)^2
\right),
\qquad M\geq M_0.
\end{equation}
Consequently,
\begin{equation}
\Gamma_M(\mathbf y)\longrightarrow\Gamma(\mathbf y).
\end{equation}
If in addition,
\begin{equation}
q_M(\mathbf y)
=
\mathcal O\!\left(h_M(\mathbf y)^2\right),
\end{equation}
then
\begin{equation}
\left\|
\Gamma_M(\mathbf y)-\Gamma(\mathbf y)
\right\|_2
=
\mathcal O\!\left(h_M(\mathbf y)^2\right).
\end{equation}
\end{proposition}

\begin{proof}
Since $\mathcal U$ is open and $\mathbf y\in\mathcal U$, choose
$\rho>0$ such that
\begin{equation}
\overline{B(\mathbf y,\rho)}
:=
\left\{
\mathbf z\in\mathbb R^d:
\|\mathbf z-\mathbf y\|_2\leq\rho
\right\}
\subset\mathcal U.
\end{equation}
Define
\begin{equation}
C_R
=
\frac{1}{2}
\sup_{\mathbf z\in\overline{B(\mathbf y,\rho)}}
\left\|
D^2\Gamma(\mathbf z)
\right\|_{\mathrm{op}}.
\end{equation}
Here,
$\|D^2\Gamma(\mathbf z)\|_{\mathrm{op}}
:=
\sup_{\|\mathbf u\|_2=\|\mathbf v\|_2=1}
\|D^2\Gamma(\mathbf z)[\mathbf u,\mathbf v]\|_2$, where
$\mathbf u,\mathbf v\in\mathbb R^d$.
The constant $C_R$ is finite because $D^2\Gamma$ is continuous and
$\overline{B(\mathbf y,\rho)}$ is compact.

Since $h_M(\mathbf y)\to0$, there exists $M_0\in\mathbb N$ such
that
\begin{equation}
h_M(\mathbf y)<\rho,
\qquad M\geq M_0.
\end{equation}
Hence, for $M\geq M_0$, every line segment joining $\mathbf y$ to
$\mathbf y_{S_M(j)}$ lies in $\overline{B(\mathbf y,\rho)}$.

Taylor's theorem therefore gives
\begin{equation}
\Gamma(\mathbf y_{S_M(j)})
=
\Gamma(\mathbf y)
+
D\Gamma(\mathbf y)
\bigl(
\mathbf y_{S_M(j)}-\mathbf y
\bigr)
+
R_{j,M},
\end{equation}
where
\begin{equation}
\left\|R_{j,M}\right\|_2
\leq
C_R
\left\|
\mathbf y_{S_M(j)}-\mathbf y
\right\|_2^2.
\end{equation}
Since $\Gamma=\Phi^{-1}$, we have
\begin{equation}
\Gamma(\mathbf y_{S_M(j)})
=
\mathbf x_{S_M(j)}.
\end{equation}
Multiplying the Taylor expansion by $b_{j,M}(\mathbf y)$, summing
over $j$, and using the definitions of $\Gamma_M$ and the convex
weights, we obtain
\begin{equation}
\begin{aligned}
\Gamma_M(\mathbf y)-\Gamma(\mathbf y)
={}&
D\Gamma(\mathbf y)
\left(
\sum_{j=1}^{K}
b_{j,M}(\mathbf y)\mathbf y_{S_M(j)}
-
\mathbf y
\right)
+
\sum_{j=1}^{K}
b_{j,M}(\mathbf y)R_{j,M}.
\end{aligned}
\end{equation}
Consequently,
\begin{equation}
\begin{aligned}
\left\|
\Gamma_M(\mathbf y)-\Gamma(\mathbf y)
\right\|_2
\leq{}&
\left\|D\Gamma(\mathbf y)\right\|_{\mathrm{op}}
q_M(\mathbf y)
+
C_R
\sum_{j=1}^{K}
b_{j,M}(\mathbf y)
\left\|
\mathbf y_{S_M(j)}-\mathbf y
\right\|_2^2
\\
\leq{}&
\left\|D\Gamma(\mathbf y)\right\|_{\mathrm{op}}
q_M(\mathbf y)
+
C_Rh_M(\mathbf y)^2.
\end{aligned}
\end{equation}
The claimed estimate follows by taking
\begin{equation}
C_{\mathbf y}
=
\max
\left\{
1,
\left\|D\Gamma(\mathbf y)\right\|_{\mathrm{op}},
C_R
\right\}.
\end{equation}

Finally, convexity of the weights gives
\begin{equation}
\begin{aligned}
q_M(\mathbf y)
&=
\left\|
\sum_{j=1}^{K}
b_{j,M}(\mathbf y)
\bigl(
\mathbf y-\mathbf y_{S_M(j)}
\bigr)
\right\|_2
\leq
\sum_{j=1}^{K}
b_{j,M}(\mathbf y)
\left\|
\mathbf y-\mathbf y_{S_M(j)}
\right\|_2
\leq
h_M(\mathbf y).
\end{aligned}
\end{equation}
Thus $h_M(\mathbf y)\to0$ also implies $q_M(\mathbf y)\to0$, and
hence
\begin{equation}
\Gamma_M(\mathbf y)\longrightarrow\Gamma(\mathbf y).
\end{equation}
The second-order conclusion follows immediately when
$q_M(\mathbf y)=\mathcal O(h_M(\mathbf y)^2)$.
\end{proof}

We note that the $K$-NN convex-interpolation strategy adopted here is not the only possible approach for approximating the pre-image map. Other numerical analysis--informed lifting strategies have been proposed, including double Diffusion Maps based on Geometric Harmonics \cite{evangelou2022double,Patsatzis_2023} and, more recently, RANDSMAPs, which provide random-feature/multiscale neural decoders with explicit mass-preservation constraints~\cite{patsatzis2026}. In the present work, we employ the $K$-NN convex interpolation approach because of its simplicity, non-intrusive character, and its favorable conservation properties for problems in which the relevant invariant is linear, such as total mass. By contrast, geometric harmonics-based pre-image approximations do not, in general, enforce such conservation constraints by construction \cite{patsatzis2026}. A detailed comparison of mass-preserving and non-mass-preserving pre-image strategies, including RANDSMAPs, is provided in~\cite{patsatzis2026}.

\subsection{Convolutional Neural Network-based Autoencoder}
\label{subsec:AE}
Autoencoders are self-supervised machine-learning models trained to map input data $\mathbf{x}_m \in \mathbb{R}^{N}$ to reconstructed outputs $\hat{\mathbf{x}}_m \in \mathbb{R}^{N}$ of the same dimensionality by forcing the information to pass through a lower-dimensional bottleneck. In this work, we employ convolutional autoencoders, which are particularly well suited for spatially distributed data such as fluid flows because of their ability to capture coherent structures through localized filters. We follow the theoretical formulation and numerical implementation presented in \cite{Fukagata_2025}.

A basic autoencoder consists of two components: an encoder and a decoder. In the present notation, these correspond to the restriction and lifting operators $\Phi_M$ and $\Gamma_M$, respectively. In contrast to other numerical analysis--based manifold-learning techniques, in autoencoders both operators are explicitly parameterized by neural networks. The encoder maps the input field $\mathbf{x}_m$ onto a latent vector,
\begin{equation}
\mathbf{y}_m = \Phi_M(\mathbf{x}_m),
\qquad \mathbf{y}_m \in \mathbb{R}^{d},
\end{equation}
where $d \ll N$ is the latent dimension. The decoder reconstructs the input from this compressed representation,
\begin{equation}
\hat{\mathbf{x}}_m = \Gamma_M(\mathbf{y}_m),
\end{equation}
so that the overall mapping reads
\begin{equation}
\hat{\mathbf{x}}_m = (\Gamma_M \circ \Phi_M)(\mathbf{x}_m).
\end{equation}
When accurate reconstruction is achieved, i.e.\ $\hat{\mathbf{x}}_m \approx \mathbf{x}_m$, the latent vector $\mathbf{y}_m$ can be interpreted as a reduced-order representation of the flow field.

It is important to note that, unlike nonlinear manifold-learning methods such as Diffusion Maps, autoencoders do not provide an intrinsic criterion for determining the latent dimension $d$, which must therefore be treated as a hyperparameter of the model. More sophisticated approaches, such as $\beta$-variational autoencoders \cite{SoleraRico2024}, attempt to address this limitation by promoting a trade-off between reconstruction accuracy and latent-space regularization, thereby encouraging parsimonious and disentangled representations. In the present work, the latent dimension is fixed for both approaches, as estimated independently of the autoencoder using DMs (see Section~\ref{sec:results}). This value is not tuned for either method: no model selection over 
$d$ was performed for the autoencoder, and the spectral gap in DM estimates the intrinsic dimension of the manifold, which need not minimize reconstruction or prediction error. Optimizing $d$ separately for each method would be substantially more expensive for the autoencoder, since a separate model must be trained for each candidate dimension, whereas DM embeddings of different dimensions follow by truncating the same eigenbasis. The comparison in Section~\ref{sec:results} is therefore performed at a common fixed dimension.

The network parameters, namely the weights and biases collected in $\mathbf{W}$, are determined by minimizing a reconstruction loss over the training dataset. In this work, we adopt the mean squared error (MSE),
\begin{equation}
\mathcal{L}(\mathbf{W})
=
\frac{1}{M}
\sum_{m=1}^{M}
\left\|
\mathbf{x}_m-\hat{\mathbf{x}}_m
\right\|_2^2,
\end{equation}
where $M$ is the number of training snapshots. The optimal parameters are defined as
\begin{equation}
\mathbf{W}^{*}
=
\arg\min_{\mathbf{W}} \mathcal{L}(\mathbf{W}).
\end{equation}

The encoder and decoder are constructed using convolutional neural networks (CNNs), which exploit the spatial structure of the input data \cite{LeCun_1998}. For the $l$-th convolutional layer, let the input tensor be
\begin{equation}
\mathbf{z}^{(l-1)} \in \mathbb{R}^{n_x^{(l-1)} \times n_y^{(l-1)} \times n_c^{(l-1)}},
\end{equation}
where $n_x^{(l-1)}$ and $n_y^{(l-1)}$ denote the spatial dimensions and $n_c^{(l-1)}$ the number of input channels. The output tensor $\mathbf{z}^{(l)} \in \mathbb{R}^{n_x^{(l)} \times n_y^{(l)} \times n_c^{(l)}}$ is computed as
\begin{equation}
c^{(l)}_{ijr}
=
\sum_{k=1}^{n_c^{(l-1)}}
\sum_{p=0}^{H-1}
\sum_{q=0}^{H-1}
w^{(l)}_{pqkr}\,
z^{(l-1)}_{i+p-G,\,j+q-G,\,k}
+
b^{(l)}_{r},
\end{equation}
\begin{equation}
z^{(l)}_{ijr}=\phi_{\mathrm{act}}\left(c^{(l)}_{ijr}\right),
\end{equation}
where $(i,j)$ denote the spatial indices of the feature maps, $k$ indexes the input channels, and $r$ the output channels. The indices $(p,q)$ refer to the position within the convolutional kernel of size $H \times H$, and $G=\lfloor H/2 \rfloor$. The coefficients $w^{(l)}_{pqkr}$ are the convolutional weights, and $b^{(l)}_{r}$ is the bias associated with the $r$-th output channel. In the present work, convolutional layers use same padding, so that the spatial dimensions are preserved across each convolution. The nonlinear activation function $\phi_{\mathrm{act}}$ is chosen as the hyperbolic tangent for all hidden layers. 

Spatial dimensionality reduction in the encoder is achieved through max-pooling layers, which halve the resolution at each stage. Conversely, the decoder employs up-sampling layers with the same factor to progressively recover the original resolution. These operations are not trainable and are used solely to control tensor dimensions while enabling hierarchical feature extraction.

At the bottleneck, the feature maps are flattened and mapped to the latent space through a fully connected (dense) layer. Denoting by
\begin{equation}
\mathbf{z}^{(l-1)} \in \mathbb{R}^{n^{(l-1)}}
\end{equation}
the flattened input vector, a dense layer produces an output $\mathbf{z}^{(l)} \in \mathbb{R}^{n^{(l)}}$ according to
\begin{equation}
c^{(l)}_{i}
=
\sum_{j=1}^{n^{(l-1)}} w^{(l)}_{ij}\,z^{(l-1)}_{j}
+
b^{(l)}_{i},
\end{equation}
\begin{equation}
z^{(l)}_{i}=\phi_{\mathrm{act}}\left(c^{(l)}_{i}\right),
\end{equation}
where $j$ indexes the input neurons, $i$ the output neurons, $w^{(l)}_{ij}$ are the weights, and $b^{(l)}_{i}$ are the biases of the dense layer. The specific CNN-based autoencoder architecture employed in this work is selected upon a systematic analysis in which the number of trainable parameters---one of the key hyperparameters in AE-based manifold learning--is varied. More details are given in Section~\ref{subsec:res_coord}.

\subsection{Proper Orthogonal Decomposition}
\label{subsec:POD}
The celebrated Proper Orthogonal Decomposition (POD) technique \cite{Lumley} decomposes the fluctuations of the velocity field $\mathbf{x}(x,y,t)$ with respect to the temporal mean $\overline{\mathbf{x}}(x,y)$, namely
\begin{equation}
\mathbf{x}'(x,y,t)=\mathbf{x}(x,y,t)-\overline{\mathbf{x}}(x,y),
\end{equation}
as
\begin{equation}
\mathbf{x}'(x,y,t)=\sum_{i=1}^{\infty} y_i(t)\,\boldsymbol{\varphi}_i(x,y),
\label{eq:POD_def}
\end{equation}
where the spatial modes $\boldsymbol{\varphi}_i(x,y)$ are orthonormal and $y_i(t)$ denotes the temporal coefficient associated with the $i$-th POD mode.
Let
\begin{equation}
\mathbf{S}'=\{\mathbf{x}'_m\}_{m=1,\ldots,M} \in \mathbb{R}^{N \times M}
\end{equation}
denote the snapshot matrix of $M$ realizations of the fluctuation field, each of dimension $N$. The discrete POD modes can be efficiently computed via the method of snapshots \cite{Sirovich}, leading to the eigenvalue problem
\begin{equation}
\mathbf{S}'^\top \mathbf{S}' \boldsymbol{\phi}_i = \sigma_i \boldsymbol{\phi}_i,
\qquad i=1,\ldots,M,
\label{eq:eigprobQQt}
\end{equation}
where $\sigma_i \ge 0$ are the eigenvalues sorted in descending order, $\sigma_1 \ge \cdots \ge \sigma_M$. The POD modes are then obtained as
\begin{equation}
\boldsymbol{\varphi}_i = \frac{\mathbf{S}' \boldsymbol{\phi}_i}{\sqrt{\sigma_i}}.
\end{equation}
By retaining the leading $d \ll M < N$ modes, one constructs a reduced-order POD basis that provides a linear parameterization of the data manifold. Using the POD basis, the restriction operator $\Phi_M$ in Eq.~\eqref{eq:HD2LDmap} is defined as
\begin{equation}
\mathbf{y}=\Phi_M(\mathbf{x})=\mathbf{W}^\top\bigl(\mathbf{x}-\overline{\mathbf{x}}\bigr),
\end{equation}
where
\begin{equation}
\mathbf{W}=[\boldsymbol{\varphi}_1,\ldots,\boldsymbol{\varphi}_d] \in \mathbb{R}^{N\times d}
\end{equation}
contains the first $d$ POD modes, and
\begin{equation}
\mathbf{y}=[y_1,\ldots,y_d]^\top
\end{equation}
contains the corresponding POD coefficients. This projection can be applied to both training snapshots and previously unseen states, thus providing latent coordinates suitable for subsequent ROM construction.
A key advantage of POD is the availability of a closed-form expression for the pre-image mapping $\Gamma_M$, enabling straightforward reconstruction of the high-dimensional field. In particular, for any latent vector $\mathbf{y}$, the lifting operator reads
\begin{equation}
\mathbf{x}=\Gamma_M(\mathbf{y})=\overline{\mathbf{x}}+\mathbf{W}\mathbf{y}.
\end{equation}
Despite its efficiency and interpretability, POD may fail to provide a geometrically consistent embedding or to identify the intrinsic dimensionality of the system. This limitation becomes critical in complex flow scenarios, such as transitions between multiple regimes or dynamics governed by codimension-2 bifurcations, where strongly nonlinear structures cannot be efficiently embedded in a linear subspace.
	
\section{Benchmark flow configuration}
\label{sec:flow_config}

\begin{figure}
    \centering
    \includegraphics[scale=0.8]{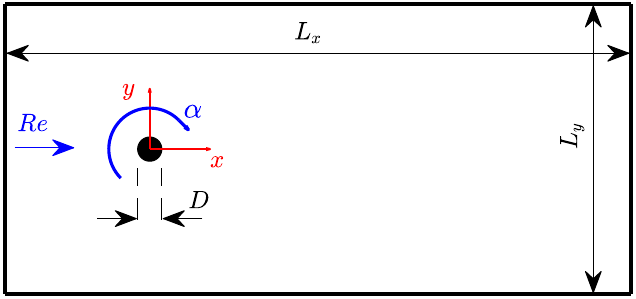}
    \caption{\label{fig:domain} Physical domain of the flow past a rotating cylinder. The governing parameters are defined as $Re = \rho U_\infty D/\mu$ and $\alpha = \Omega D / (2 U_\infty)$.}
\end{figure}

For our illustration, we focus on the incompressible two-dimensional flow past a rotating circular cylinder (see Fig.~\ref{fig:domain}). The flow is controlled by two parameters: the Reynolds number $Re = \rho U_\infty D/\mu$ and the rotation rate $\alpha = \Omega D / (2 U_\infty)$. Here, $\Omega$ is the dimensional angular velocity of the cylinder, $U_\infty$ is the free-stream velocity, $D$ is the cylinder diameter, $\mu$ is the dynamic viscosity, and $\rho$ is the constant fluid density. In the following, we consider clockwise rotation of the cylinder surface, corresponding to $\alpha > 0$.

The flow motion inside the domain is governed by the two-dimensional incompressible Navier--Stokes equations,

\begin{subequations}
\begin{equation}
\dfrac{\partial u}{\partial x} + \dfrac{\partial v}{\partial y} = 0,
\label{eq:continuity}
\end{equation}
\begin{equation}
\dfrac{\partial u}{\partial t} + u \dfrac{\partial u}{\partial x} + v \dfrac{\partial u}{\partial y}
= -\dfrac{\partial p}{\partial x} + \dfrac{1}{Re}\left(\dfrac{\partial^2 u}{\partial x^2} + \dfrac{\partial^2 u}{\partial y^2}\right),
\label{eq:momentum_u}
\end{equation}
\begin{equation}
\dfrac{\partial v}{\partial t} + u \dfrac{\partial v}{\partial x} + v \dfrac{\partial v}{\partial y}
= -\dfrac{\partial p}{\partial y} + \dfrac{1}{Re}\left(\dfrac{\partial^2 v}{\partial x^2} + \dfrac{\partial^2 v}{\partial y^2}\right),
\label{eq:momentum_v}
\end{equation}
\end{subequations}

where $u \equiv u(x,y,t)$ and $v \equiv v(x,y,t)$ denote the streamwise and cross-stream velocity components, respectively, defined over the two-dimensional spatial domain $(x,y)$ and evolving in time $t$. Note that Eqs.~\eqref{eq:continuity}--\eqref{eq:momentum_v} are written in dimensionless form, with velocity and length scaled by the reference quantities $U_\infty$ and $D$, respectively. Time is scaled by $D/U_\infty$, and pressure by $\rho U_\infty^2$.

The computational domain is a rectangle excluding the interior of the cylinder, with sides of length $L_x = 26D$ and $L_y = 12D$. The cylinder center is located at the streamwise position $x = 6D$. At the domain inlet (the left side of the rectangle), a uniform velocity profile is prescribed, namely $\mathbf{u} \equiv (u,v) = (1,0)$. The same values are also assigned as initial conditions to start the computations. A standard free-outflow boundary condition is enforced at the domain outlet (the right side of the rectangle), while the top and bottom boundaries are equipped with homogeneous Neumann boundary conditions for all variables. On the cylinder surface, no-slip boundary conditions with prescribed surface rotation are imposed such that
\begin{equation}
\mathbf{u}\cdot\mathbf{t} = \alpha,
\qquad
\mathbf{u}\cdot\mathbf{n} = 0,
\end{equation}
where $\mathbf{t}$ and $\mathbf{n}$ denote the tangential and normal unit vectors to the surface in the $(x,y)$ plane, respectively.

A uniform structured grid is used to discretize the physical domain, with mesh spacing $\Delta x = \Delta y = 0.05$.
This resolution corresponds to 20 grid cells per characteristic length $D$, resulting in a total of $N_g = N_x \times N_y = 124{,}800$ grid points. Such a spatial resolution is required to accurately capture the influence of cylinder rotation on the downstream evolution of the flow. Numerical simulations are conducted over the parameter range $Re \in [40,80]$ and $\alpha \in [4.5,6.5]$. The streamwise ($u$) and cross-stream ($v$) velocity components are sampled with a time step $\Delta t = 0.2$ over a total simulation time $T = 200$. As a result, 1000 temporal realizations of the velocity field are obtained for each combination of the ($Re$,$\alpha$) parameters. More details on the numerical simulation of Eqs.~\eqref{eq:continuity}--\eqref{eq:momentum_v} and on the generation of the datasets used to train and test the different manifold-learning methods employed in this work are given in Appendix~\ref{app:NS_eq}.

Snapshots of the simulations are shown in Fig.~\ref{fig:2D_maps} in terms of instantaneous two-dimensional contour maps of the cross-stream velocity component $v$ for different values of $Re$ and $\alpha$: the Reynolds number increases from left to right, while the rotation rate increases from top to bottom. Moreover, Fig.~\ref{DNS_var_t} reports the time evolution of the velocity signal $v(\bar{x},\bar{y},t)$ recorded at the spatial location $(\bar{x},\bar{y}) = (5,0)$ as the Reynolds number increases from $Re=40$ to $Re=80$ at fixed $\alpha = 5.5$ (panel (a)), and as the rotation rate increases from $\alpha=4.5$ to $\alpha=6.5$ at fixed $Re=70$ (panel (b)).

For relatively low Reynolds numbers ($Re = 40$ and $50$), the flow remains steady for all values of the rotation rate $\alpha$ (see panels (a), (d), and (g) in Fig.~\ref{fig:2D_maps}, and the black and red curves in Fig.~\ref{DNS_var_t}(a)). The cross-stream velocity component $v$ becomes uniform in the far wake, while its spatial distribution around the cylinder reflects the imposed clockwise rotation ($\alpha > 0$): positive ($v > 0$) upstream of the cylinder and negative ($v < 0$) downstream, consistently with the induced circulation.

\begin{figure}
    \centering
    \includegraphics[scale=0.8]{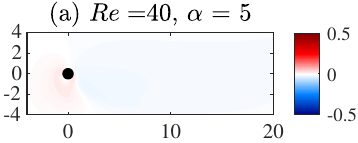}
    \includegraphics[scale=0.8]{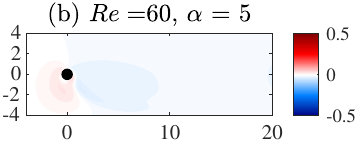}
    \includegraphics[scale=0.8]{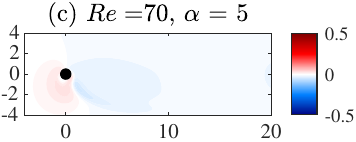}\\
    \includegraphics[scale=0.8]{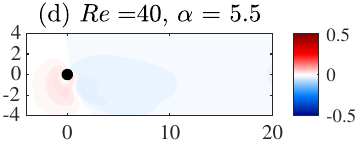}
    \includegraphics[scale=0.8]{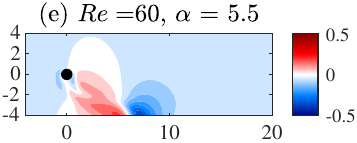}
    \includegraphics[scale=0.8]{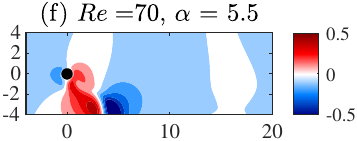}\\
    \includegraphics[scale=0.8]{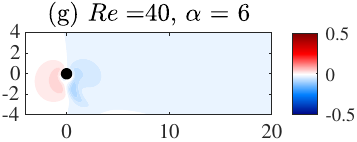}
    \includegraphics[scale=0.8]{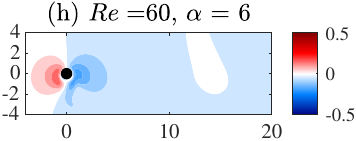}
    \includegraphics[scale=0.8]{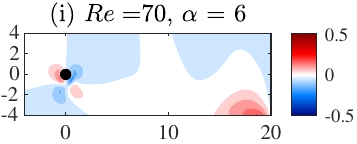}\\
    \caption{\label{fig:2D_maps} Navier--Stokes velocity-field snapshots (the $v$ component) for different values of the governing parameters $(Re,\alpha)$ in the high-dimensional physical space $x$--$y$.}
\end{figure}

\begin{figure}
    \centering
    \includegraphics[scale=0.8]{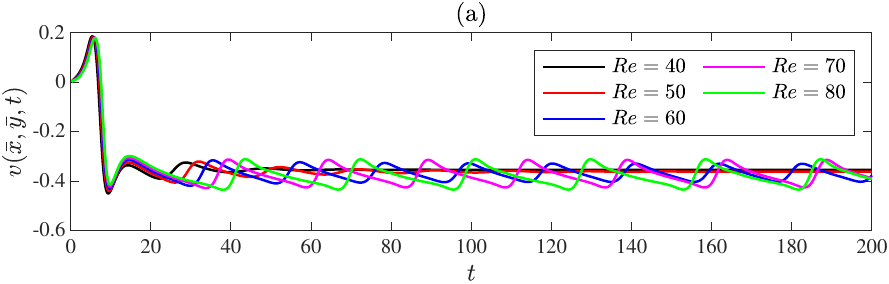} \\
    \vspace{0.2cm}
    \includegraphics[scale=0.8]{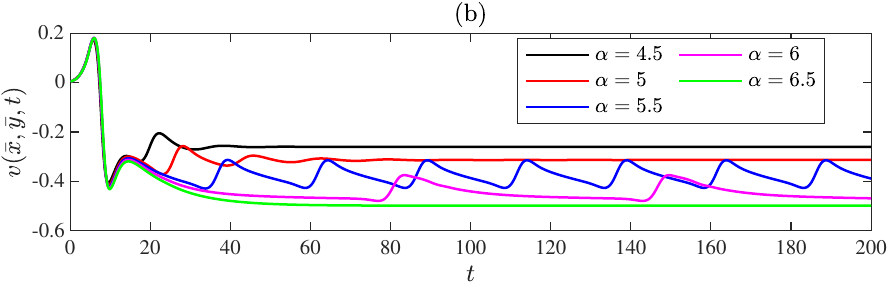}
    \caption{\label{DNS_var_t} Time evolution of the Navier--Stokes velocity field (the $v$ component) at the spatial location $(\bar{x},\bar{y}) = (5,0)$ by varying $Re$ at fixed $\alpha = 5.5$ (a), and by varying $\alpha$ at fixed $Re = 70$ (b).}
\end{figure}

As the Reynolds number increases to $Re = 60$, the flow undergoes a qualitative change, and three distinct regimes emerge as $\alpha$ is varied. For $\alpha \leq 5$, the flow remains steady (Fig.~\ref{fig:2D_maps}(b)). In the intermediate range $5 < \alpha < 6$, the flow becomes unsteady (Fig.~\ref{fig:2D_maps}(e)), with the onset of time-periodic vortex shedding, as shown by the blue curve in Fig.~\ref{DNS_var_t}(a). This regime corresponds to the emergence of a stable limit cycle in the continuous NS equations, consistent with the crossing of an Andronov--Hopf bifurcation as $\alpha$ is increased. Finally, for $\alpha \geq 6$, the flow returns to a steady state (Fig.~\ref{fig:2D_maps}(h)), indicating a restabilization of the steady solution.

The unsteady regime is characterized by a downward-deflected vortex street, resulting from the combined effects of inertia and rotation, which break the symmetry of the wake. This asymmetric shedding persists and becomes more pronounced as the Reynolds number is further increased to $Re = 70$ and $80$ (see panels (c), (f), and (i) in Fig.~\ref{fig:2D_maps}, and the magenta and green curves in Fig.~\ref{DNS_var_t}(a)). In this regime, the system exhibits a limit cycle whose properties depend sensitively on $\alpha$. In particular, the shedding period increases monotonically with $\alpha$, as evidenced by the comparison between the blue and magenta curves in Fig.~\ref{DNS_var_t}(b).

The progressive increase in the oscillation period is a key dynamical signature of this flow configuration and reflects the underlying bifurcation structure in the $(Re,\alpha)$ parameter space. As the rotation rate approaches $\alpha = 6.5$, the period diverges and the oscillations are suppressed, giving rise to a new steady regime (green curve in Fig.~\ref{DNS_var_t}(b)). This behavior is characteristic of a global homoclinic bifurcation in which the limit cycle collides with a saddle equilibrium. At this point, the trajectory spends increasingly long times in the vicinity of the saddle, leading to the observed divergence of the period before the cycle ultimately disappears.

Overall, this sequence of transitions can be interpreted within the framework of a codimension-2 Bogdanov--Takens bifurcation in the $(Re,\alpha)$ parameter space, where Hopf, saddle-node, and homoclinic bifurcation curves collide~\cite{Sierra}. In this scenario, the Hopf bifurcation gives rise to the limit cycle, the saddle-node bifurcation governs the appearance of multiple steady solutions, and the homoclinic bifurcation marks the disappearance of the limit cycle through its interaction with a saddle point. Such a bifurcation scenario, involving the coexistence of multiple local and global bifurcations, makes the selected flow configuration a stringent benchmark for reduced-order modelling approaches, highlighting the need for nonlinear manifold-learning techniques, as will be demonstrated in Section~\ref{sec:results}.

\section{Results}
\label{sec:results}
The numerical framework introduced in Section~\ref{sec:methodology} is here applied to construct surrogate reduced-order models (ROMs) for the discrete-time maps of the latent dynamics of the rotating-cylinder flow. 
We consider a set of spatio-temporal snapshots $\{\mathbf{x}_m\}_{m=1,\ldots,M} \subset \mathbb{R}^N$, obtained from numerical simulations of the Navier--Stokes equations (Eqs.~\eqref{eq:continuity}--\eqref{eq:momentum_v}). Each snapshot represents the discretized velocity field (both $u$ and $v$ components) over the spatial domain.

We first examine the identification of the reduced-order coordinates (Section~\ref{subsec:res_coord}), comparing the three manifold-learning techniques considered in this work: POD, Parsimonious Diffusion Maps (PDMs), and autoencoders (AEs). All details regarding the generation of the datasets used to train and test the different methods are given in Appendix~\ref{app:NS_eq}. The resulting POD-, PDMs-, and AE-based ROMs are then employed for long-term prediction of the latent dynamics. Finally, the ROM solutions are reconstructed in the high-dimensional space, enabling a direct comparison with the corresponding Navier--Stokes flow fields. Such a comparison is quantified by the reconstruction error $\varepsilon_m$ associated with each snapshot $m=1, \dots, M$, as defined in Eq.~\eqref{eq:reconstruction_error} in Appendix~\ref{app:NS_eq}. 

Note that, since we compare different methods in this section, we refer to the manifold coordinates $\mathbf{y}$ obtained with the different methods using different symbols, while remaining consistent with the generic latent coordinate $\mathbf{y}$ introduced in Section~\ref{sec:methodology}. In particular, we denote $\mathbf{y} = \boldsymbol{\psi}$ for PDMs, $\mathbf{y} = \mathbf{a}$ for POD, and $\mathbf{y} = \boldsymbol{\gamma}$ for AE.

\subsection{Reduced-order coordinate identification}
\label{subsec:res_coord}

\begin{figure}
    \centering
    \includegraphics[scale=0.8]{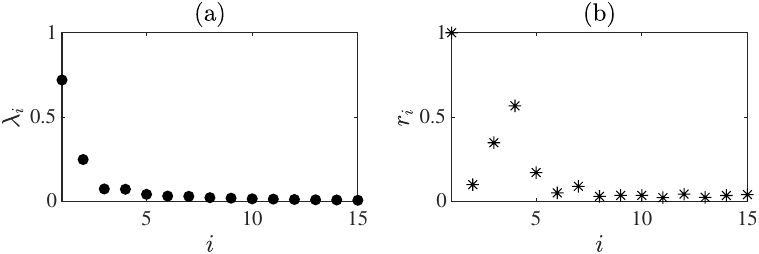}\\
    \vspace{0.2cm}
    \includegraphics[scale=0.8]{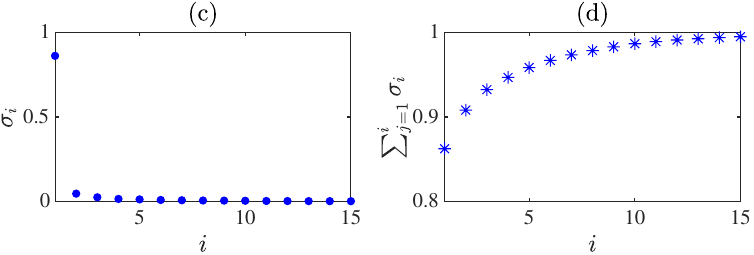}
    \caption{\label{DMs-POD_spectrum} Diffusion Maps eigenvalue spectrum $\lambda_i$ (a) and corresponding local linear fitting coefficients $r_i$ for the parsimonious selection (b) as functions of the mode index $i$. In (c) and (d), the corresponding POD eigenvalues $\sigma_i$ and their cumulative sum are reported, respectively.}
\end{figure}

The Diffusion Maps eigenvalues $\lambda_i$ and the corresponding local linear fitting coefficients $r_i$ are shown in Fig.~\ref{DMs-POD_spectrum}(a)--(b), respectively, for $i = 1, \dots, 15$. The POD eigenvalues $\sigma_i$ and their cumulative sum are reported in Fig.~\ref{DMs-POD_spectrum}(c)--(d), respectively.

A key result is that Diffusion Maps provide a more rigorous and physically meaningful identification of the latent space dimension compared to POD. The POD spectrum (Fig.~\ref{DMs-POD_spectrum}(c)) does not exhibit a clear separation of scales: while the first mode is dominant, the remaining modes decay gradually without a distinct spectral gap. As a result, standard selection criteria become ambiguous. For instance, a gap-based criterion would suggest retaining only the first mode, which is clearly insufficient to represent the system dynamics, whereas an energy-based criterion (e.g., $99\%$ cumulative energy) requires approximately 15 modes, leading to a significantly less parsimonious representation. Moreover, the cumulative energy distribution (Fig.~\ref{DMs-POD_spectrum}(d)) is inherently monotonic, and therefore does not naturally support a sharp truncation.
In contrast, the DMs spectrum displays a clear separation between the first five modes and the remaining ones, as evidenced by the presence of a pronounced knee in the eigenvalues $\lambda_i$ for $i \geq 5$ (Fig.~\ref{DMs-POD_spectrum}(a)). This separation is further corroborated by the behavior of the local linear fitting coefficients $r_i$, which are significantly larger for the first five modes, while $r_i \approx 0$ for the others. Together, these indicators consistently identify the first five modes as the relevant degrees of freedom governing the latent flow dynamics. Accordingly, the choice of five DMs modes is directly supported by both the spectral structure and the residual distribution. This choice yields a parsimonious yet accurate representation of the flow over the entire $(Re, \alpha)$ parameter space, as will be shown in the remainder of this section.

These observations highlight a key advantage of nonlinear manifold learning via DMs: it enables a more rigorous identification of the intrinsic latent dimension of fluid flows than classical POD-based approaches, while also promoting a compact and efficient reduced-order representation. Moreover, such a latent representation is physically meaningful, as the leading DMs coordinates $\psi_i$ ($i = 1, \dots, 5$) are directly related to the physical mechanisms underlying transitions among the different flow regimes induced by variations in the governing parameters $Re$ and $\alpha$, as shown in Fig.~\ref{DMs_embedding_dynamics}. For all values of $Re$, the latent coordinates exhibit a stationary steady state for both low ($\alpha = 4.5$, Fig.~\ref{DMs_embedding_dynamics}(a)--(c)) and high ($\alpha = 6.5$, Fig.~\ref{DMs_embedding_dynamics}(g)--(i)) rotation rates. Note that the two steady states are clearly distinguished by the values of $\psi_1$ reached at long times $t$: $\psi_1 > 0$ for $\alpha = 4.5$ and $\psi_1 < 0$ for $\alpha = 6.5$. These correspond to the two different high-dimensional stable steady solutions existing in this range of $\alpha$, as discussed in Section~\ref{sec:flow_config}. Moreover, for $\alpha = 5.5$, the latent dynamics exhibit a transition from asymptotically steady conditions at $Re = 45$ (Fig.~\ref{DMs_embedding_dynamics}(d)) to a stable limit cycle at $Re = 75$ (Fig.~\ref{DMs_embedding_dynamics}(f)), with periodic oscillations captured across all five parsimonious DMs coordinates. The transition from steady to periodic behavior occurs through a long-time steady regime that exhibits damped oscillations, observed at $Re = 55$ (Fig.~\ref{DMs_embedding_dynamics}(e)). This is an important feature and, as we will see later in this section, surrogate models based on the DMs coordinates can reproduce such a regime, whereas POD-based ROMs may fail.

\begin{figure}
    \centering
    \includegraphics[scale=0.8]{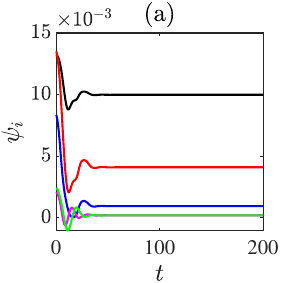}
    \includegraphics[scale=0.8]{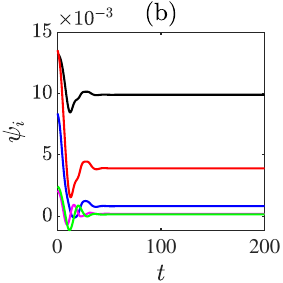}
    \includegraphics[scale=0.8]{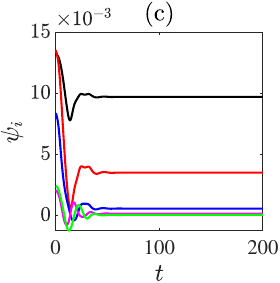}
    \includegraphics[scale=0.8]{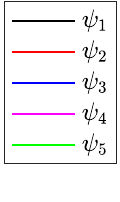}\\
    \vspace{0.2cm}
    \includegraphics[scale=0.8]{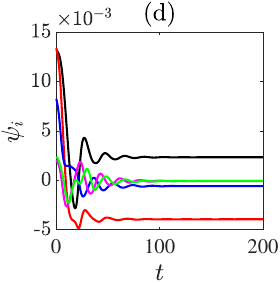}
    \includegraphics[scale=0.8]{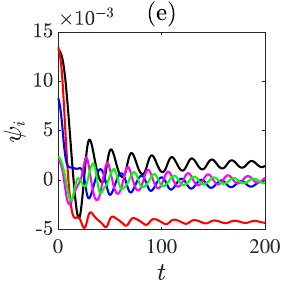}
    \includegraphics[scale=0.8]{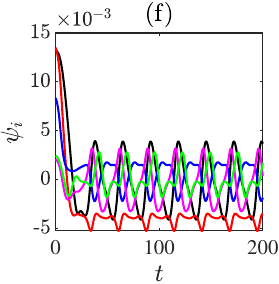}
    \includegraphics[scale=0.8]{Figures/DMs_embedding_legend}\\
    \vspace{0.2cm}
    \includegraphics[scale=0.8]{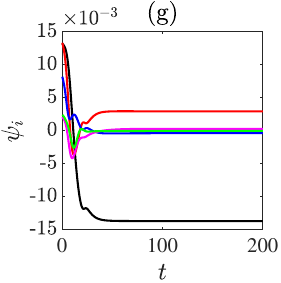}
    \includegraphics[scale=0.8]{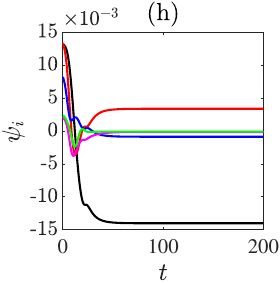}
    \includegraphics[scale=0.8]{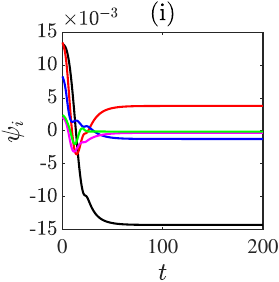}
    \includegraphics[scale=0.8]{Figures/DMs_embedding_legend}
    \caption{\label{DMs_embedding_dynamics} Low-dimensional (latent) dynamics identified by the PDMs coordinates $\psi_i$ ($i=1, \dots, 5$) as functions of time $t$ for different values of $\alpha$ (from top to bottom) and $Re$ (from left to right): $\alpha = 4.5$ ((a)--(c)); $\alpha = 5.5$ ((d)--(f)); $\alpha = 6.5$ ((g)--(i)); $Re = 45$ ((a), (d), and (g)); $Re = 55$ ((b), (e), and (h)); $Re = 75$ ((c), (f), and (i)).}
\end{figure}

We now turn to the autoencoder-based embedding used for comparison with PDMs and POD. To select an appropriate architecture, we trained six different convolutional autoencoder configurations, AE$_1$ through AE$_6$, sharing the same overall encoder--decoder architecture (see Section~\ref{subsec:AE}) but with a progressively increasing number of trainable parameters. Fig.~\ref{Manifold_baseline_recon_error_AUTOENCODERS} and Table~\ref{tab:autoencoder_error_statistics} report, respectively, the resulting baseline reconstruction error $\varepsilon_m$ and the associated training cost and parameter count for each of the six architectures, evaluated on both the training and test datasets.

\begin{figure}
    \centering
    \includegraphics[scale=0.8]{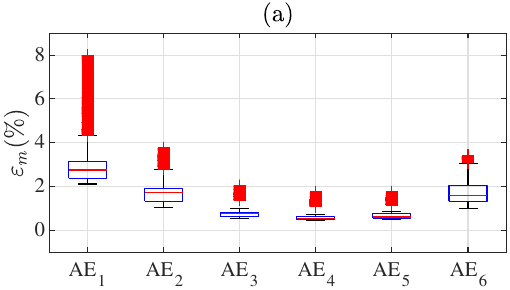}
    \hspace{0.2cm}
    \includegraphics[scale=0.8]{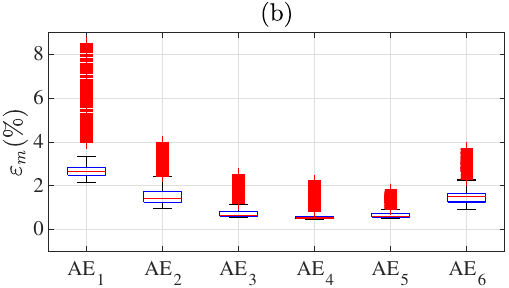}
    \caption{\label{Manifold_baseline_recon_error_AUTOENCODERS} Baseline reconstruction error $\varepsilon_m$ of six different AE embeddings (AE$_1$--AE$_6$; see Table~\ref{tab:autoencoder_error_statistics}) with increasing numbers of trainable parameters, evaluated over $m = 1, \dots, 4500$ snapshots obtained for different $Re$ and $\alpha$ values sampled from the training (a) and test (b) datasets. In each box plot, the red line indicates the median, the blue box edges correspond to the 25$^{\text{th}}$ and 75$^{\text{th}}$ percentiles, the black whiskers extend to the most extreme values within 1.5 times the inter-quartile range, and the red markers denote values beyond this range.}
\end{figure}
\begin{table}
    \centering
    \setlength{\tabcolsep}{5pt}
    {\rowcolors{1}{white}{cyan!8}
        \begin{tabularx}{\textwidth}{lCCCC}
            \toprule
            \textbf{Embedding}
            & \textbf{Training time (CPU-hours)}
            & \textbf{Trainable parameters}
            & $\boldsymbol{\varepsilon_m}$ (train dataset)
            & $\boldsymbol{\varepsilon_m}$ (test dataset) \\
            \midrule
            AE$_1$ & 19.43 & 1803 & $2.98\,(2.11,\,5.47)\,\%$ & $2.97\,(2.15,\,5.78)\,\%$ \\
            AE$_2$ & 30.32 & 6471 & $1.71\,(1.05,\,3.24)\,\%$ & $1.69\,(0.97,\,3.45)\,\%$ \\
            AE$_3$ & 69.73 & 28603 & $0.82\,(0.54,\,1.66)\,\%$ & $0.86\,(0.55,\,1.78)\,\%$ \\
            \rowcolor{green!20}
            AE$_4$ & 79.69 & 275735 & $0.64\,(0.45,\,1.31)\,\%$ & $0.70\,(0.46,\,1.43)\,\%$ \\
            AE$_5$ & 239.36 & 1091751 & $0.71\,(0.50,\,1.40)\,\%$ & $0.75\,(0.48,\,1.52)\,\%$ \\
            AE$_6$ & 384.62 & 4344647 & $1.66\,(0.99,\,2.89)\,\%$ & $1.62\,(0.90,\,3.08)\,\%$ \\
            \bottomrule
        \end{tabularx}
    }
    \caption{Summary of the baseline reconstruction error $\varepsilon_m$ of six different AE embeddings with increasing numbers of trainable parameters (reported in the third column) on the train (fourth column) and test (fifth column) datasets. The error $\varepsilon_m$ is shown in terms of the average over $m = 1, \dots, 4500$ snapshots obtained for different $Re$ and $\alpha$ values, together with the 5--95\% percentiles. In the second column, the total training computational cost of each AE embedding is reported. The row shaded in green highlights the selected AE embedding, which is used for comparison with POD and PDMs throughout the remainder of this section.}
    \label{tab:autoencoder_error_statistics}
\end{table}

Moving from AE$_1$ to AE$_4$, the reconstruction error consistently decreases on both datasets, indicating that increasing the number of trainable parameters initially enhances the capability of the autoencoder to approximate the underlying data manifold. In particular, AE$_4$ achieves the lowest reconstruction error among all the tested architectures, with average errors of $0.64\%$ on the training set and $0.70\%$ on the test set, as also reflected by the compact distributions shown in Fig.~\ref{Manifold_baseline_recon_error_AUTOENCODERS} (fourth column in panels (a)--(b)). However, further increasing the model complexity (AE$_5$ and AE$_6$) does not yield additional improvements. On the contrary, a degradation in performance is observed, especially for AE$_6$, where both training and test errors increase significantly. This behavior indicates a degradation of both the learning and generalization performance of the autoencoder: increasing the number of trainable parameters beyond a certain point makes the optimization problem harder to solve effectively, leading to a deterioration of the learned manifold representation that is reflected in both training and test errors.

These results highlight the delicate trade-off between model expressiveness and trainability in autoencoder-based manifold learning. While sufficiently expressive architectures (i.e.\ with a relatively high number of learnable weights) are required to capture the intrinsic structure of the data, overly complex models may become increasingly difficult to train effectively, ultimately degrading performance. Based on this analysis, AE$_4$ was selected as the optimal autoencoder architecture, as it provides the best compromise between reconstruction accuracy and learning and generalization performance. Unless otherwise stated, all AE results reported in the remainder of the paper refer to this selected architecture. Notably, this configuration has approximately $2.8 \times 10^5$ trainable parameters and required nearly $80$ CPU-hours of training time (fourth row in Table~\ref{tab:autoencoder_error_statistics}), i.e.\ about three orders of magnitude higher than the computational cost associated with POD and PDMs.

Table~\ref{tab:cnn_ae} details the resulting architecture of the selected AE model. The encoder consists of six convolutional layers with a progressively decreasing number of filters, each followed by a max-pooling operation, resulting in a progressive reduction of the spatial resolution. As discussed in Appendix~\ref{app:NS_eq}, the input data to the autoencoder consist of the two-dimensional velocity field defined over a physical domain discretized on a $128 \times 128$ grid. Accordingly, the input tensor has size $(128,128,2)$, where the two channels correspond to the streamwise ($u$) and cross-stream ($v$) velocity components. The successive pooling operations reduce the spatial dimensions as $(128,128) \rightarrow (64,64) \rightarrow (32,32) \rightarrow (16,16) \rightarrow (8,8) \rightarrow (4,4) \rightarrow (2,2)$, while the number of channels evolves according to the number of filters in each convolutional layer. After the last convolutional layer, the tensor has size $(2,2,8)$, corresponding to $n^{(l-1)} = 2 \times 2 \times 8 = 32$ degrees of freedom. The encoder output tensor is flattened into a vector of dimension $n^{(l-1)} = 32$, which is then mapped onto the latent space of dimension $d=5$, i.e.\ $\mathbf{y}_m = \mathbf{z}^{(l)} \in \mathbb{R}^d$. The decoder mirrors this structure through a dense projection, reshaping, and successive upsampling and convolutional layers, ultimately reconstructing the original two-component field. In particular, the latent vector $\mathbf{y}_m \in \mathbb{R}^d$ with $d=5$ is first mapped back to a vector of dimension $n^{(l)} = 32$, which is then reshaped into a tensor of size $(2,2,8)$ and progressively upsampled to recover the original resolution. All convolutional layers use $3\times3$ kernels with same padding. Hidden layers employ the hyperbolic tangent activation function, while the final output layer is linear in order to reconstruct continuous-valued quantities.

\begin{table}
\centering
\caption{Best-performing CNN-based autoencoder architecture used in the present work.}
\label{tab:cnn_ae}
\begin{tabular}{l c l c}
\toprule
\multicolumn{2}{c}{\textbf{Encoder}} & \multicolumn{2}{c}{\textbf{Decoder}} \\
\cmidrule(r){1-2} \cmidrule(l){3-4}
\textbf{Layer} & \textbf{Data size} & \textbf{Layer} & \textbf{Data size} \\
\midrule
Input & $(128,128,2)$ & Latent vector & $(5)$ \\
Conv2D $(3\times3,128)$ & $(128,128,128)$ & Dense & $(32)$ \\
MaxPooling $(2\times2)$ & $(64,64,128)$ & Reshape & $(2,2,8)$ \\
Conv2D $(3\times3,64)$ & $(64,64,64)$ & UpSampling $(2\times2)$ & $(4,4,8)$ \\
MaxPooling $(2\times2)$ & $(32,32,64)$ & Conv2D $(3\times3,8)$ & $(4,4,8)$ \\
Conv2D $(3\times3,64)$ & $(32,32,64)$ & UpSampling $(2\times2)$ & $(8,8,8)$ \\
MaxPooling $(2\times2)$ & $(16,16,64)$ & Conv2D $(3\times3,16)$ & $(8,8,16)$ \\
Conv2D $(3\times3,32)$ & $(16,16,32)$ & UpSampling $(2\times2)$ & $(16,16,16)$ \\
MaxPooling $(2\times2)$ & $(8,8,32)$ & Conv2D $(3\times3,32)$ & $(16,16,32)$ \\
Conv2D $(3\times3,16)$ & $(8,8,16)$ & UpSampling $(2\times2)$ & $(32,32,32)$ \\
MaxPooling $(2\times2)$ & $(4,4,16)$ & Conv2D $(3\times3,64)$ & $(32,32,64)$ \\
Conv2D $(3\times3,8)$ & $(4,4,8)$ & UpSampling $(2\times2)$ & $(64,64,64)$ \\
MaxPooling $(2\times2)$ & $(2,2,8)$ & Conv2D $(3\times3,64)$ & $(64,64,64)$ \\
Flatten & $(32)$ & UpSampling $(2\times2)$ & $(128,128,64)$ \\
Dense (latent) & $(5)$ & Conv2D $(3\times3,128)$ & $(128,128,128)$ \\
 &  & Output Conv2D $(3\times3,2)$ & $(128,128,2)$ \\
\bottomrule
\end{tabular}
\end{table}

Having selected and characterized the best-performing AE architecture, we now compare it against POD and PDMs. Fig.~\ref{Manifold_baseline_recon_error} and Table~\ref{tab:rom_metrics_manifold} report the baseline reconstruction error $\varepsilon_m$ evaluated over $m = 1, \dots, 4500$ snapshots obtained for different $Re$ and $\alpha$ values sampled from the training and test datasets, for all three methods.

\begin{figure}
    \centering
    \includegraphics[scale=0.8]{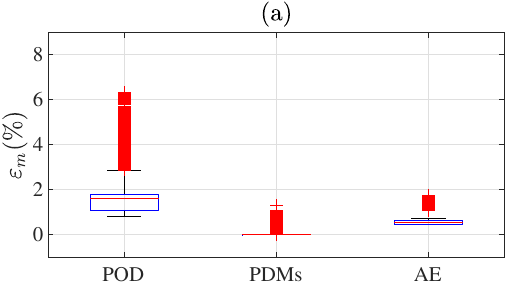}
    \hspace{0.2cm}
    \includegraphics[scale=0.8]{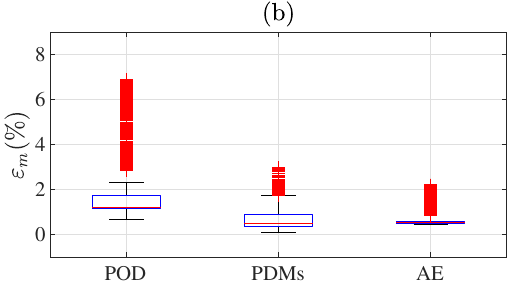}
    \caption{\label{Manifold_baseline_recon_error} Baseline reconstruction error $\varepsilon_m$ of the POD (first column), parsimonious DMs (second column), and AE (third column) embeddings, evaluated over $m = 1, \dots, 4500$ snapshots obtained for different $Re$ and $\alpha$ values sampled from the training (a) and test (b) datasets. In each box plot, the red line indicates the median, the blue box edges correspond to the 25$^{\text{th}}$ and 75$^{\text{th}}$ percentiles, the black whiskers extend to the most extreme values within 1.5 times the inter-quartile range, and the red markers denote values beyond this range. The AE results correspond to the best-performing architecture identified in Fig.~\ref{Manifold_baseline_recon_error_AUTOENCODERS} and Table~\ref{tab:autoencoder_error_statistics}.}
\end{figure}

\begin{table}
    \centering
    \setlength{\tabcolsep}{5pt}
    {\rowcolors{1}{white}{cyan!8}
        \begin{tabularx}{\textwidth}{lCCC}
            \toprule
            \textbf{Embedding}
            & \textbf{Training time (CPU-hours)}
            & $\boldsymbol{\varepsilon_m}$ (train dataset)
            & $\boldsymbol{\varepsilon_m}$ (test dataset) \\
            \midrule
            POD  & 0.06 & $1.77\,(0.83,5.06)\,\%$ & $1.69\,(0.69,5.40)\,\%$ \\
            PDMs & 0.06  & $0.03\,(0.00,0.19)\,\%$ & $0.70\,(0.22,1.61)\,\%$ \\
            AE   & 79.69 & $0.64\,(0.45,1.31)\,\%$ & $0.70\,(0.46,1.43)\,\%$ \\
            \bottomrule
        \end{tabularx}
    }
    \caption{Summary of the performance of the manifold-learning methods on the train (third column) and test (fourth column) datasets. The baseline reconstruction error $\varepsilon_m$ is shown in terms of the average over $m = 1, \dots, 4500$ snapshots obtained for different $Re$ and $\alpha$ values, together with the 5--95\% percentiles. In the second column, the total training computational cost for POD, PDMs, and the selected AE embedding is reported. The AE row corresponds to the best-performing architecture identified in Fig.~\ref{Manifold_baseline_recon_error_AUTOENCODERS} and Table~\ref{tab:autoencoder_error_statistics}.}
    \label{tab:rom_metrics_manifold}
\end{table}

A first quantitative comparison between POD-based and PDMs-based latent representations of the rotating-cylinder flow dynamics is provided in Fig.~\ref{Manifold_baseline_recon_error}, which reports the baseline reconstruction error $\varepsilon_m$ evaluated over $m = 1, \dots, 4500$ snapshots obtained for different $Re$ and $\alpha$ values sampled from the training (panel (a)) and test (panel (b)) datasets. This metric represents the error obtained after performing the restriction and lifting stages, without introducing any surrogate model for the latent dynamics. Therefore, it isolates the representational capability of the learned low-dimensional manifold, which constitutes a fundamental prerequisite for the construction of accurate predictive ROMs. The same error obtained using the selected autoencoder embedding is also reported (third column in Fig.~\ref{Manifold_baseline_recon_error}(a)--(b)).

Since DMs with parsimonious selection (PDMs) identify an intrinsic latent dimension equal to $d=5$, the comparison among the three manifold-learning methods is performed by fixing the number of modes to $d=5$ for all three methods. It can be clearly appreciated that PDMs consistently outperform POD on both datasets, achieving significantly lower reconstruction errors and reduced variability. On the training set (Fig.~\ref{Manifold_baseline_recon_error}(a)), PDMs yield nearly negligible errors, with a median close to zero and a very narrow inter-quartile range. This is quantitatively confirmed in Table~\ref{tab:rom_metrics_manifold}, where the average error is $0.03\%$, compared with $1.77\%$ for POD. On the test dataset (Fig.~\ref{Manifold_baseline_recon_error}(b)), PDMs maintain a clear advantage over POD, with an average error of $0.70\%$ versus $1.69\%$. Although the spread of the PDMs error increases compared with the training case, the distribution remains significantly more concentrated than that of POD, indicating improved generalization to unseen data.

The comparison between PDMs and AE embeddings is particularly insightful. Despite the substantially higher training cost (approximately $80$ CPU-hours, compared with $0.06$ for both POD and PDMs; see Table~\ref{tab:rom_metrics_manifold}), the AE does not provide a clear advantage over PDMs. On the training dataset, PDMs achieve even lower reconstruction errors than AE ($0.03\%$ vs.\ $0.64\%$ on average), highlighting the effectiveness of Diffusion Maps in capturing the intrinsic geometry of the data. On the test dataset, the two methods exhibit comparable performance (both around $0.70\%$ average error), suggesting that PDMs can match the generalization capability of neural-network-based approaches at negligible computational cost.

This further emphasizes the advantage of Diffusion Maps, which achieve comparable---or even superior---performance without the need for extensive hyperparameter tuning or high computational resources. Indeed, the training cost of AE should be regarded as a lower bound on the effective computational demand, since the sensitivity of the learned embedding to the random initialization of the network weights introduces an additional source of variability that must be accounted for in practice. A dedicated analysis over $N_{\text{seed}} = 100$ independent training runs, provided in Appendix~\ref{app:UQ_AE}, confirms that while AE is robust across initializations, the spread in reconstruction error is non-negligible in the periodic flow regimes, further motivating the use of geometry-preserving alternatives such as Diffusion Maps when computational efficiency and reproducibility are primary concerns.

\subsection{Forecasting in the low-dimensional space}
\label{subsec:res_forecast}
\begin{figure}
    \centering
    \includegraphics[scale=0.8]{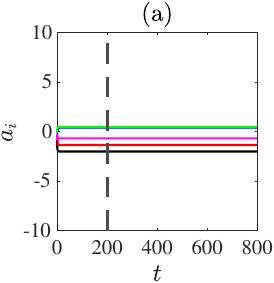}
    \includegraphics[scale=0.8]{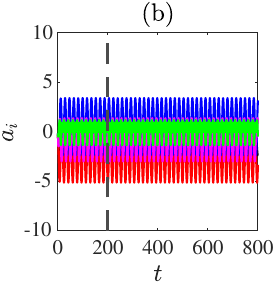}
    \includegraphics[scale=0.8]{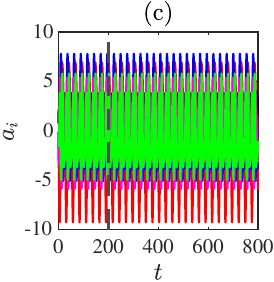}
    \includegraphics[scale=0.8]{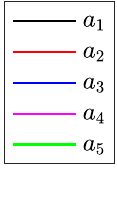}\\
    \vspace{0.2cm}
    \includegraphics[scale=0.8]{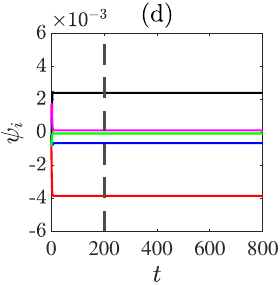}
    \includegraphics[scale=0.8]{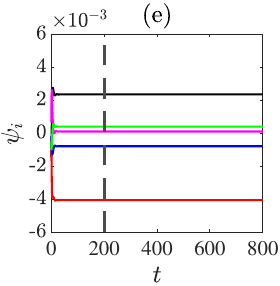}
    \includegraphics[scale=0.8]{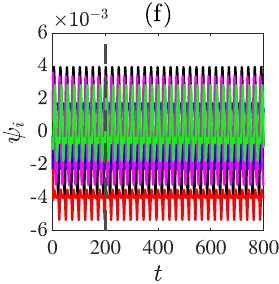}
    \includegraphics[scale=0.8]{Figures/DMs_embedding_legend}\\
    \vspace{0.2cm}
    \includegraphics[scale=0.8]{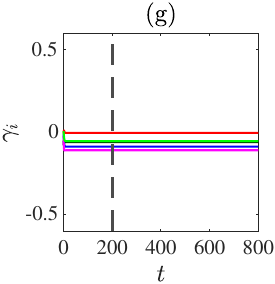}
    \includegraphics[scale=0.8]{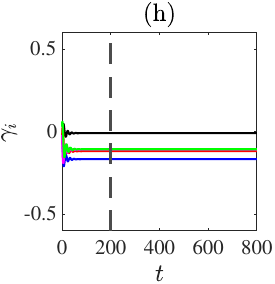}
    \includegraphics[scale=0.8]{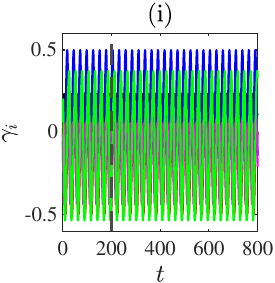}
    \includegraphics[scale=0.8]{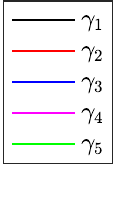}
    \vspace{0.2cm}
    \caption{\label{predicted_dynamics} Long-term prediction of the low-dimensional (latent) dynamics identified by the POD ((a)--(c)), PDMs ((d)--(f)), and AE ((g)--(i)) reduced-order models for $\alpha = 5.5$ at different values of the Reynolds number: $Re = 45$ (left panels), $Re = 55$ (central panels), and $Re = 75$ (right panels). The vertical black dashed line denotes the time horizon used for ROM training.}
\end{figure}

\begin{figure}
    \centering
    \includegraphics[scale=0.8]{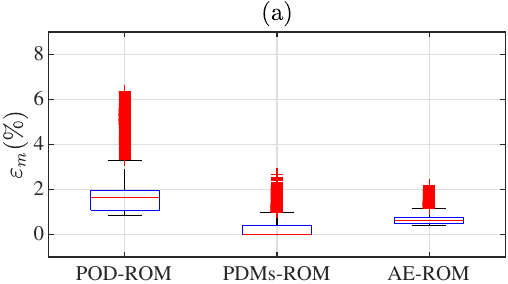}
    \hspace{0.2cm}
    \includegraphics[scale=0.8]{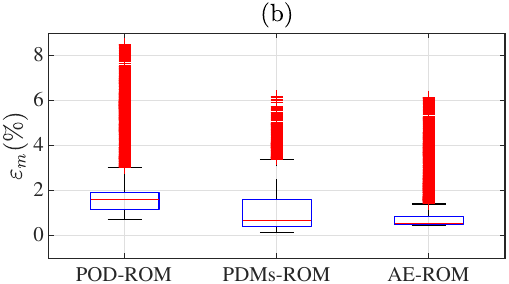}
    \caption{\label{ROM_prediction_recon_error} Prediction reconstruction error $\varepsilon_m$ (i.e.\ after solution of the pre-image problem on the forecast latent dynamics) of the POD (first column), PDMs (second column), and AE (third column) reduced-order models, evaluated over $m = 1, \dots, 4500$ snapshots obtained for different $Re$ and $\alpha$ values sampled from the training (a) and test (b) datasets. In each box plot, the red line indicates the median, the blue box edges correspond to the 25$^{\text{th}}$ and 75$^{\text{th}}$ percentiles, the black whiskers extend to the most extreme values within 1.5 times the inter-quartile range, and the red markers denote values beyond this range.}
\end{figure}

The identification of linear (via POD) and nonlinear (via PDMs and AE) embeddings of the rotating-cylinder flow dynamics enables the construction of surrogate models in the second stage of the procedure. 
The POD-, PDMs-, and AE-based ROMs are learned as discrete time-maps of the solution operator in the identified latent space via Gaussian Process Regression (see Eq.~\eqref{eq:GPR_y} in Appendix~\ref{app:GPR} and related discussion), and are numerically evaluated to assess their forecasting performance in the respective manifold coordinates; the results of this analysis are reported in Fig.~\ref{predicted_dynamics}. For all cases, the time integration is carried out up to a final time $t = 800$, which is four times larger than the training time horizon ($t = 200$, highlighted by the vertical black dashed line in all panels of Fig.~\ref{predicted_dynamics}). Moreover, both the initial conditions in the reduced-coordinate space and the parameters $Re$ and $\alpha$ are drawn from the test set, meaning that they were not observed during the training phase of the ROMs.

The PDMs-based ROM faithfully reproduces the transition from the steady regime, characterized by a fixed-point attractor in phase space at $Re \leq 55$ (Fig.~\ref{predicted_dynamics}(d)--(e)), to a limit-cycle regime at $Re = 75$ (Fig.~\ref{predicted_dynamics}(f)). In contrast, comparison with the POD-based ROMs (Fig.~\ref{predicted_dynamics}(a)--(c)) reveals an important limitation: while the steady (panel (a)) and limit-cycle (panel (c)) regimes are correctly reproduced, the transition between them is misidentified (panel (b)). In particular, the damped oscillations that should asymptotically converge to the fixed-point attractor (see Fig.~\ref{DMs_embedding_dynamics}(e) in Section~\ref{subsec:res_coord}) are incorrectly learned as a spurious limit cycle (see Fig.~\ref{predicted_dynamics}(b)). This highlights a key advantage of PDMs over POD-based ROMs, namely their ability to correctly identify all dynamical regimes, even in the presence of strongly nonlinear behaviors induced by the codimension-2 bifurcation occurring in this flow configuration (see the discussion in Section~\ref{sec:flow_config}). Fig.~\ref{predicted_dynamics}(g)--(i) shows the latent dynamics predicted by the AE-ROM for the same values of $Re$ and $\alpha$ discussed above. The AE latent coordinates reproduce the same dynamical regimes observed with PDMs. However, as discussed in Section~\ref{subsec:res_coord}, identifying an accurate AE embedding requires approximately three orders of magnitude more computational time than PDMs, as well as careful tuning of the hyperparameters of the convolutional neural networks used in the encoding and decoding stages to avoid overfitting (see Table~\ref{tab:autoencoder_error_statistics} and Fig.~\ref{Manifold_baseline_recon_error_AUTOENCODERS}).

\subsection{Reconstruction in the high-dimensional space}
\label{subsec:res_lift}

The third and final stage of the procedure consists of lifting (or decoding) the low-dimensional solutions forecast on the identified manifold back to the high-dimensional space, in order to evaluate the reconstruction error with respect to the ground-truth Navier--Stokes solutions. This stage requires solving the so-called pre-image problem in manifold learning. While, in the case of POD, the mapping between reduced- and full-order spaces is available in closed form, for nonlinear manifold-learning techniques the pre-image problem is generally ill-posed, and several approaches have been proposed in the literature (see the related discussion in the Introduction). In the present work, we employ a $K$-nearest-neighbors ($K$-NN) strategy with $K=4$ coupled with PDMs, whereas for the AE the decoding stage is naturally performed by the convolutional neural network used in the decoder.

The prediction reconstruction errors $\varepsilon_m$ (i.e.\ computed after solving the pre-image problem on the forecast latent dynamics) for the POD-, PDMs-, and AE-based ROMs are reported in Fig.~\ref{ROM_prediction_recon_error} and Table~\ref{tab:rom_metrics_prediction}. These metrics are evaluated over $m = 1, \dots, 4500$ snapshots obtained for different values of $Re$ and $\alpha$, sampled from both the training (Fig.~\ref{ROM_prediction_recon_error}(a)) and test (Fig.~\ref{ROM_prediction_recon_error}(b)) datasets. It can be seen that the same trend observed for the baseline reconstruction error (see Fig.~\ref{Manifold_baseline_recon_error} and Table~\ref{tab:rom_metrics_manifold} in Section~\ref{subsec:res_coord}) is preserved: PDMs-ROMs consistently outperform POD-ROMs on both training and test datasets, while achieving accuracy comparable to AE-ROMs. More specifically, on the training dataset, PDMs-ROMs yield a significantly lower average error ($0.27\%$) than POD-ROMs ($1.81\%$), with a much tighter error distribution, as indicated by the narrower percentile range (second column in Fig.~\ref{ROM_prediction_recon_error}(a)). Remarkably, AE-ROMs also improve over POD-ROMs ($0.72\%$), but are slightly less accurate than PDMs-ROMs in terms of mean error (third column in Fig.~\ref{ROM_prediction_recon_error}(a)). On the test dataset (Fig.~\ref{ROM_prediction_recon_error}(b)), the gap between PDMs-ROMs and AE-ROMs becomes marginal, with nearly identical average errors ($1.09\%$ and $1.06\%$, respectively), both substantially lower than that of POD-ROMs ($2.14\%$). It is also worth noting that PDMs-ROMs exhibit a lower minimum error (approaching zero on the training set), indicating their capability to achieve near-exact reconstructions under favorable conditions.

\begin{figure}
    \centering
    \includegraphics[scale=0.8]{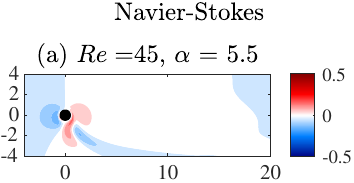}
    \includegraphics[scale=0.8]{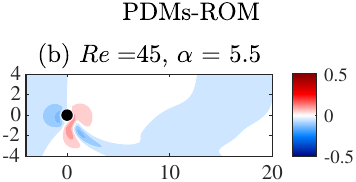}
    \includegraphics[scale=0.8]{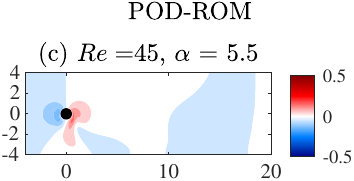}\\
    \vspace{0.2cm}
    \includegraphics[scale=0.8]{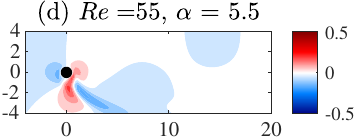}
    \includegraphics[scale=0.8]{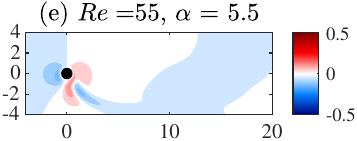}
    \includegraphics[scale=0.8]{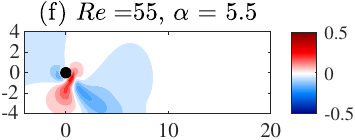}\\
    \vspace{0.2cm}
    \includegraphics[scale=0.8]{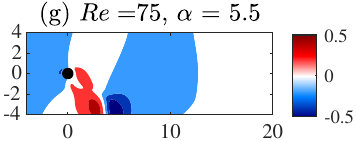}
    \includegraphics[scale=0.8]{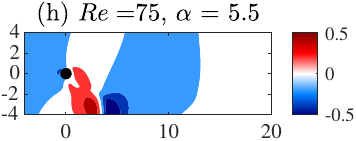}
    \includegraphics[scale=0.8]{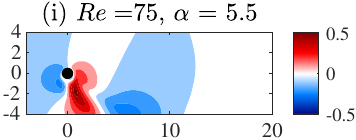}\\
    \caption{\label{predicted_contours} Comparison among the ground-truth Navier--Stokes $v$ velocity component field (left panels), the PDMs-ROM prediction (central panels), and the POD-ROM prediction (right panels) at the selected time instant $t = 200$ for different values of the governing parameters $(Re, \alpha)$ in the high-dimensional physical space $x$--$y$.}
\end{figure}

Finally, Fig.~\ref{predicted_contours} presents a comparison between the ground-truth Navier--Stokes $v$ velocity component field (left panels), the PDMs-ROM prediction (central panels), and the POD-ROM prediction (right panels) at the selected time instant $t = 200$, for different values of the governing parameters $(Re, \alpha)$. These cases span all the dynamical regimes of the considered flow configuration in the high-dimensional physical space. It can be clearly observed that the PDMs-ROM is able to faithfully reproduce the local variations of the steady Navier--Stokes velocity field induced by the cylinder rotation ($\alpha = 5.5$, $Re \leq 55$, Fig.~\ref{predicted_contours}(a)--(b) and (d)--(e)), as well as the asymmetric periodic vortex-shedding regime established at $Re = 75$ (Fig.~\ref{predicted_contours}(g)--(h)). In contrast, POD-based reconstructions are only able to qualitatively capture the velocity field, exhibiting significant errors in both the near- and far-field regions, in both the steady and periodic regimes (Fig.~\ref{predicted_contours}(c), (f), and (i)). We explicitly note that the AE-based high-dimensional velocity-field reconstructions have been found to be indistinguishable from those obtained with PDMs and therefore are not reported.

Overall, these results highlight PDMs as a computationally efficient, interpretable, and highly accurate nonlinear manifold-learning approach for reduced-order modelling of complex, multi-regime flow dynamics. In particular, PDMs enable the identification of the appropriate latent-space dimension, which is critical in configurations characterized by multiple dynamical regimes. This is especially relevant in the present case, where the underlying continuous-time NS system exhibits a codimension-2 Bogdanov--Takens bifurcation. Within this framework, PDMs not only significantly outperform POD-based ROMs in terms of reconstruction accuracy, but also achieve the same level as AE-based models in terms of predictive performance. However, unlike AEs, DMs provide a transparent and physically interpretable embedding, where the latent dimension emerges naturally from the data. In contrast, AEs rely on a largely black-box architecture, requiring the \emph{a priori} specification of the latent dimension, as well as training on the order of $10^5$ learnable parameters. This leads to a computational cost that, for the present flow configuration, is three orders of magnitude higher than that of PDMs, in addition to the need for careful hyperparameter tuning to avoid overfitting.

\begin{table}
    \centering
    \setlength{\tabcolsep}{5pt}
    {\rowcolors{1}{white}{cyan!8}
        \begin{tabularx}{\textwidth}{lCC}
            \toprule
            \textbf{Model}
            & $\boldsymbol{\varepsilon_m}$ (train dataset)
            & $\boldsymbol{\varepsilon_m}$ (test dataset) \\
            \midrule
            POD-ROM  & $1.81\,(0.83,5.11)\,\%$ & $2.14\,(0.71,5.99)\,\%$ \\
            PDMs-ROM & $0.27\,(0.00,1.31)\,\%$ & $1.09\,(0.25,3.29)\,\%$ \\
            AE-ROM   & $0.72\,(0.45,1.44)\,\%$ & $1.06\,(0.46,3.88)\,\%$ \\
            \bottomrule
        \end{tabularx}
    }
    \caption{Summary of the reduced-order-model performance on the train (first column) and test (second column) datasets. Each column reports the prediction reconstruction error $\varepsilon_m$ (i.e.\ after solution of the pre-image problem on the forecast latent dynamics) in terms of the average over $m = 1, \dots, 4500$ snapshots obtained for different $Re$ and $\alpha$ values, together with the 5--95\% percentiles.}
    \label{tab:rom_metrics_prediction}
\end{table}

\section{Conclusions}
\label{sec:conclusions}

This work revisits nonlinear manifold learning as an alternative to autoencoders for data-driven reduced-order modelling of fluid flows. We have shown that Parsimonious Diffusion Maps (PDMs) identify intrinsic latent coordinates directly from data, estimating their dimension rather than treating it as a hyperparameter, and that the resulting coordinates support accurate reconstruction and prediction of the full-order flow field.

The proposed framework combines PDMs for latent-space identification, Gaussian Process Regression (GPR) for learning the latent dynamics, and $K$-nearest-neighbors ($K$-NN) interpolation for approximating the pre-image, or lifting, map. This yields a fully non-intrusive ROM that provides uncertainty quantification through GPR and enables reconstruction of the high-dimensional flow fields from the predicted latent states. Importantly, the use of local interpolation for the pre-image problem offers a simple, transparent, and consistent reconstruction strategy, avoiding the need to train an additional nonlinear decoder.

The methodology has been assessed on the two-dimensional flow past a rotating cylinder, a benchmark configuration exhibiting multiple dynamical regimes governed by the Reynolds number $Re$ and the rotation rate $\alpha$. In the parameter space considered, the Navier--Stokes equations undergo a sequence of transitions associated with a codimension-2 Bogdanov--Takens bifurcation, involving the interaction of Andronov--Hopf, saddle-node, and homoclinic bifurcations. This provides a stringent benchmark for reduced-order modelling approaches.

The results demonstrate that PDMs-based ROMs significantly outperform classical linear approaches based on Proper Orthogonal Decomposition (POD), both in terms of reconstruction accuracy and predictive capability across different dynamical regimes. At the same time, PDMs achieve performance comparable to convolutional autoencoder (AE)-based ROMs at a small fraction of the computational cost, requiring only the selection of the kernel bandwidth and the parsimonious cutoff rather than a search over network architectures and training hyperparameters. Moreover, PDMs provide a transparent and physically interpretable embedding, in which the latent dimension emerges naturally from the data.

A key outcome of this work is that the PDMs-based latent representation provides a faithful, interpretable, and low-dimensional description of the different dynamical regimes of the flow, even in the presence of complex bifurcation structures. This property, together with the significantly lower computational cost of Diffusion Maps and Nystr\"om extension compared with autoencoder-based approaches, establishes Parsimonious Diffusion Maps as a powerful alternative to black-box autoencoder-based ROMs for complex, multi-regime fluid flows. In particular, the ability of PDMs to identify intrinsic coordinates and to clearly separate distinct dynamical regimes makes them well suited for future developments at the interface of data-driven modelling, numerical bifurcation and stability analysis, control, and real-time prediction, including their integration with classical numerical continuation techniques in the presence of global and codimension-2 bifurcations.

For the dataset sizes typical of the high-fidelity CFD simulations considered here, PDMs thus offer a favorable balance between accuracy, interpretability, and computational efficiency. A practical limitation, however, is their memory demand: the standard construction requires forming a kernel or affinity matrix whose size grows with the number of snapshots. Autoencoders, by contrast, can be trained in mini-batches and may therefore prove more scalable for substantially larger datasets than those considered here.

\appendix

\section{Numerical simulation of Navier--Stokes equations and dataset construction}
\label{app:NS_eq}

The Navier--Stokes equations~\eqref{eq:continuity}--\eqref{eq:momentum_v} are solved using a projection method, whereby a provisional velocity field is first computed via the second-order accurate Bell--Colella--Glaz (BCG) advection scheme, neglecting the pressure gradient, and subsequently projected onto the space of divergence-free fields by solving a pressure Poisson equation. Details on the numerical implementation in the open-source code \texttt{BASILISK} can be found in Popinet~\cite{Popinet2003}.

As described in Section~\ref{sec:flow_config}, the physical domain is discretized using a uniform structured grid with mesh spacing $\Delta x = \Delta y = 0.05$, corresponding to 20 grid cells per characteristic length $D$. This results in a total of $N_g = N_x \times N_y = 124{,}800$ grid points. Numerical simulations are performed over the parameter ranges $Re \in [40,80]$ and $\alpha \in [4.5,6.5]$, with sampling steps $\Delta Re = 5$ and $\Delta \alpha = 0.25$. The velocity field is sampled with a time step $\Delta t = 0.2$ over a total simulation time $T = 200$, yielding $1000$ temporal snapshots for each $(Re,\alpha)$ configuration.

At each time instant, the flow is described by the two-dimensional velocity field $\mathbf{u}(x,y,t) = (u(x,y,t), v(x,y,t))$, defined over the spatial domain. Upon spatial discretization, each snapshot is represented as a high-dimensional vector $\mathbf{x}_m \in \mathbb{R}^N$, obtained by stacking the two velocity components at all grid points. This notation is adopted consistently throughout the paper, in agreement with Section~\ref{sec:methodology}.

To construct the mapping $\Phi_M$ in Eq.~\eqref{eq:HD2LDmap} using PDMs, POD, and autoencoders, a subset of the data is selected from the parameter grid defined by $Re \in \{40,50,60,70,80\}$ and $\alpha \in \{4.5,5.0,5.5,6.0,6.5\}$, resulting in $25$ parameter combinations. For each $(Re,\alpha)$ pair, only $250$ snapshots are retained, randomly sampled from the last $500$ time steps of each simulation, in order to focus on statistically stationary regimes of the flow rather than transient dynamics.

Furthermore, to reduce the computational cost associated with the construction of the POD snapshot matrix (Eq.~\eqref{eq:eigprobQQt}), the Markov matrix for PDMs (Eq.~\eqref{eq:Markov}), and the training of the autoencoders, the velocity fields are spatially down-sampled via bilinear interpolation onto a coarser grid of size $\tilde{N}_x \times \tilde{N}_y = 16384$. This results in a reduced state dimension
$N = 2 \tilde{N}_x \tilde{N}_y = 32768$,
accounting for the two velocity components.

\begin{figure}
    \centering
    \includegraphics[scale=0.95]{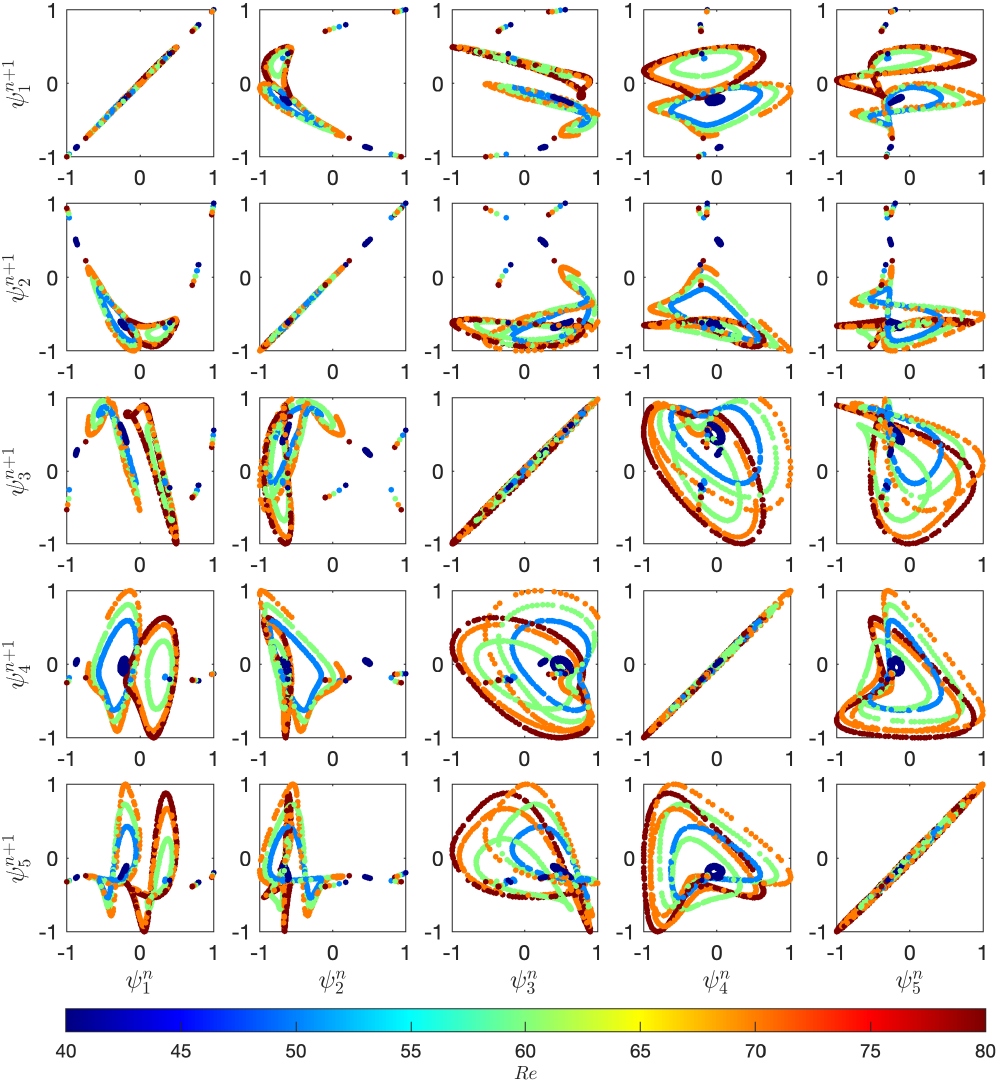}
\caption{\label{Training_data_varRE}Numerical dataset used to train the PDMs-based ROM of the latent dynamics of the rotating cylinder flow, expressed in PDMs coordinates $\psi_i$ ($i = 1, \dots, 5$) and colored by the Reynolds number $Re$. Here, $\psi_i^{n} \equiv \psi_i(t^n)$ and $\psi_i^{n+1} \equiv \psi_i(t^n + T_p)$, where $T_p = 0.2$ is the prediction time horizon. Note that the PDMs coordinates $\psi_i$ are normalized to the range $[-1, 1]$.}
\end{figure}

The resulting training dataset consists of $M = 6250$ snapshots $\mathbf{x}_m \in \mathbb{R}^N$ ($m=1, \dots, M$), which are used both to construct the restriction operator $\Phi_M$ (via PDMs, POD, and autoencoders) and to learn the latent dynamics through Gaussian Process Regression (see Appendix~\ref{app:GPR}).
The test set comprises all remaining parameter combinations in the original grid (i.e., the $56$ configurations not included in the training set) and is used exclusively to assess the generalization capability of the learned reduced-order models.

An overview of the PDMs training dataset is reported in Figs.~\ref{Training_data_varRE}--\ref{Training_data_varALPHA}, where the five Parsimonious Diffusion Maps coordinates $\mathbf{y} \equiv (\psi_1, \dots, \psi_5)$ selected for training are colored by the Reynolds number $Re$ and the rotation rate $\alpha$, respectively. For brevity, the corresponding datasets used to train the POD- and AE-based models are not shown. The data distribution exhibits a rich and intricate geometric structure. 
\begin{figure}
    \centering
    \includegraphics[scale=0.95]{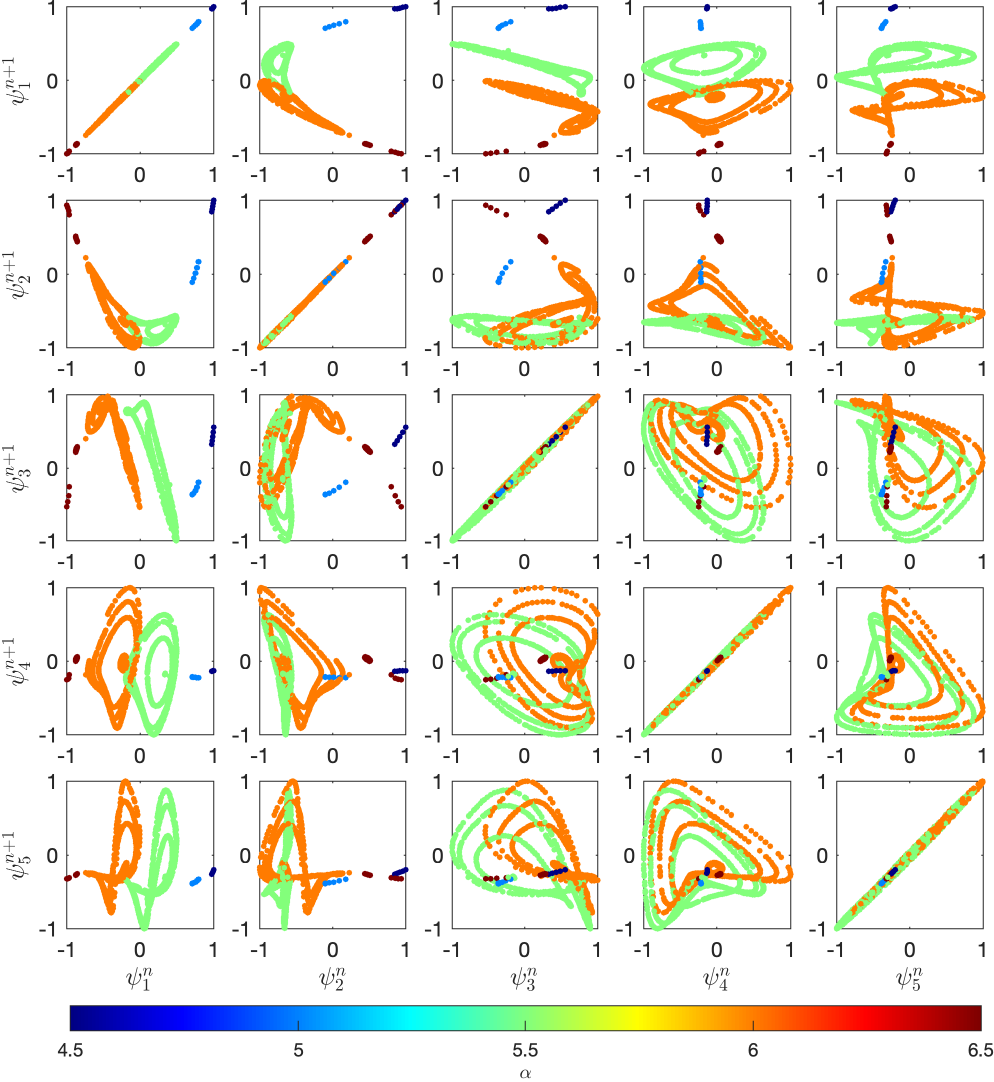}
\caption{\label{Training_data_varALPHA}Numerical dataset used to train the PDMs-based ROM of the latent dynamics of the rotating cylinder flow, expressed in PDMs coordinates $\psi_i$ ($i = 1, \dots, 5$) and colored by the rotation rate $\alpha$. Here, $\psi_i^{n} \equiv \psi_i(t^n)$ and $\psi_i^{n+1} \equiv \psi_i(t^n + T_p)$, where $T_p = 0.2$ is the prediction time horizon. Note that the PDMs coordinates $\psi_i$ are normalized to the range $[-1, 1]$.}
\end{figure}
In the $d = 5$ parsimonious DMs embedding, the data lie on curved manifolds with visible twisting and folding patterns, reflecting the coexistence and interaction of multiple time scales.
These complex structures are direct manifestations of the dynamical transitions associated with the underlying bifurcation scenario. The latent manifold is clearly not well approximated by a linear subspace; rather, it exhibits a non-trivial geometry and topology that require a nonlinear parametrization.
This observation further justifies the use of PDMs for this configuration: linear approaches such as POD would neither provide a geometrically consistent embedding nor reliably identify the intrinsic dimensionality required to capture the codimension-2 Bogdanov--Takens bifurcation and the associated complex dynamics. In contrast, the PDMs coordinates yield a smooth and dynamically meaningful embedding on which the latent evolution can be effectively learned.

For POD, the training snapshots are assembled into the snapshot matrix used in Eq.~\eqref{eq:eigprobQQt}. For PDMs, they are used to construct the Markov matrix in Eq.~\eqref{eq:Markov}. For autoencoders, they form the dataset used for training the encoder--decoder network. The same training data are also used to construct time-shifted pairs $\bigl(\mathbf{x}_m(t), \mathbf{x}_m(t+T_p)\bigr)$, which are mapped to the latent space and employed to train the Gaussian Process Regression models for the latent dynamics.

To evaluate the reconstruction error, we consider two subsets of $4500$ snapshots each, extracted from the training and test datasets, respectively. The training subset is constructed from the parameter combinations $Re \in \{40,50,70\}$ and $\alpha \in \{4.5,5.5,6.5\}$, retaining $500$ snapshots for each pair, while the test subset is constructed analogously from $Re \in \{45,55,75\}$ and $\alpha \in \{4.5,5.5,6.5\}$.

The reconstruction error associated with a given snapshot $\mathbf{x}_m$, drawn from either the training or test set, is defined as
\begin{equation}
\varepsilon_m
=
\frac{
\|\mathbf{x}_m - \hat{\mathbf{x}}_m\|_2
}{
\|\mathbf{x}_m\|_2
}
\times 100,
\label{eq:reconstruction_error}
\end{equation}
where $\hat{\mathbf{x}}_m = \Gamma_M(\Phi_M(\mathbf{x}_m))$ denotes the reconstructed velocity field. This metric provides a relative error measure, expressed as a percentage, quantifying the accuracy of the reduced-order representation.

\section{ROMs via Gaussian Process Regression}
\label{app:GPR}

Once a reduced-coordinate embedding of the flow dynamics has been obtained via the restriction operator $\Phi_M$ (using POD, PDMs, or AE), the latent dynamics are approximated by identifying the unknown solution operator introduced in Eq.~\eqref{eq:regression_basic_map} (Section~\ref{sec:methodology}) using Gaussian Process Regression (GPR). For consistency with the notation adopted in Section~\ref{sec:methodology}, we denote the latent variables by $\mathbf{y}(t) \in \mathbb{R}^d$. The same formulation applies regardless of the dimensionality-reduction technique employed (POD, PDMs, or AE). The predictors in Eq.~\eqref{eq:regression_basic_map} are taken to be the latent state itself, i.e.\
\begin{equation}
\mathbf{r}(t) = \mathbf{y}(t) \in \mathbb{R}^d,
\end{equation}
while the fixed parameters $\boldsymbol{\chi} \in \mathbb{R}^2$ include the two governing parameters of the flow configuration considered in Section~\ref{sec:flow_config}, namely the Reynolds number $Re$ and the rotation rate $\alpha$.

Accordingly, we construct ROMs with input vector
\begin{equation}
\mathbf{z}(t) = [\mathbf{y}(t),\, Re,\, \alpha]^\top \in \mathbb{R}^{d+2},
\end{equation}
and outputs given by the predicted latent state $\mathbf{y}(t+T_p)$ over a finite prediction horizon $T_p$.

The GPR surrogate models approximate each component of the latent mapping as
\begin{equation}
y_i(t+T_p) = g_i\bigl(\mathbf{y}(t), Re, \alpha; \boldsymbol{\theta}_i\bigr) + e_i,
\qquad e_i \sim \mathcal{N}(0,\sigma_i^2),
\quad i=1,\ldots,d,
\label{eq:reg_GPs_discrete}
\end{equation}
where $\mathbf{y} = [y_1,\ldots,y_d]^\top$ and $\sigma_i^2$ denotes the noise variance of the $i$-th component (not to be confused with the POD eigenvalues $\sigma_i$ of Section~\ref{subsec:POD}). Each component
\begin{equation}
g_i(\mathbf{z};\boldsymbol{\theta}_i) = g_i([\mathbf{y}, Re, \alpha]^\top;\boldsymbol{\theta}_i)
\end{equation}
is modelled as an independent Gaussian Process
\begin{equation}
g_i \sim \mathcal{GP}\bigl(0, k_i(\mathbf{z},\mathbf{z}' \mid \boldsymbol{\theta}_i)\bigr).
\end{equation}

For the implementation of GPR, we consider the latent representations $\mathbf{y}_m$ of the $M$ observations, together with the corresponding parameter values $(Re_m,\alpha_m)$, collected in the matrix
\begin{equation}
\mathbf{Z} = [\mathbf{z}_1, \dots, \mathbf{z}_M]^\top \in \mathbb{R}^{M \times (d+2)},
\qquad \mathbf{z}_m = [\mathbf{y}_m, Re_m, \alpha_m]^\top.
\end{equation}
Let $\mathbf{Y}_i = [y_{i,1}^+, \dots, y_{i,M}^+]^\top \in \mathbb{R}^M$ denote the target outputs corresponding to the $i$-th latent component across all training points, where
\begin{equation}
y_{i,m}^+ := y_i(t_m+T_p).
\end{equation}

The prior distribution for the $i$-th component is
\begin{equation}
P(\mathbf{g}_i \mid \mathbf{Z}) = \mathcal{N}\bigl(\mathbf{g}_i \mid \mathbf{0}, \mathbf{K}_i(\mathbf{Z}, \mathbf{Z} \mid \boldsymbol{\theta}_i)\bigr),
\end{equation}
where $\mathbf{g}_i = [g_i(\mathbf{z}_1), \dots, g_i(\mathbf{z}_M)]^\top$.

\begin{figure} 
\centering \includegraphics[scale=0.8]{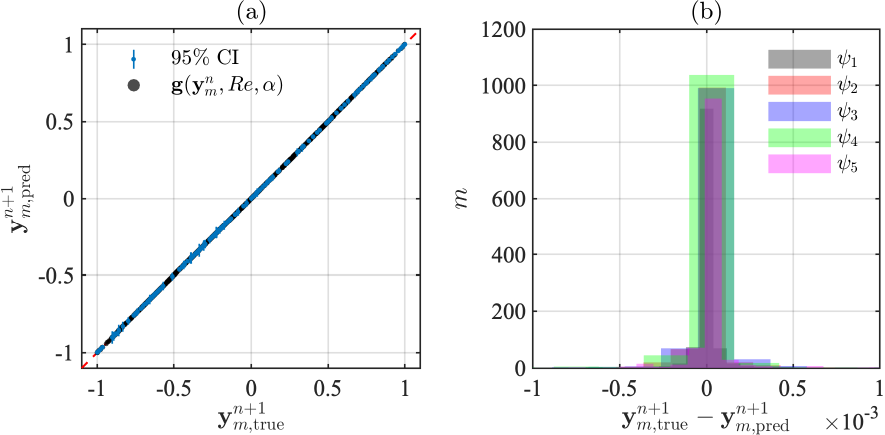} \caption{Uncertainty quantification of the PDMs-based ROM learned via Gaussian Process Regression (GPR). For input data $(\mathbf{y}^n_m, Re, \alpha)$ selected from the test set, the model predictions $\mathbf{y}^{n+1}_{m,\text{pred}}$ are reported in terms of their expected (mean) values (black dots) and corresponding 95\% confidence intervals (cyan bands), and compared against the true values $\mathbf{y}^{n+1}_{m,\text{true}}$ (panel (a), where the red dashed line denotes the bisector). Panel (b) shows the distribution of the prediction error $\mathbf{y}^{n+1}_{m,\text{true}} - \mathbf{y}^{n+1}_{m,\text{pred}}$ over all test samples, reported separately for each PDMs coordinate $\psi_i$ ($i = 1, \dots, 5$) normalized to the range $\psi_i \in [-1, 1]$.}
\label{fig:uncertainty} 
\end{figure}

Predictions at a new input $\mathbf{z}_* \in \mathbb{R}^{d+2}$ are obtained from the joint Gaussian distribution
\begin{equation}
\begin{bmatrix}
\mathbf{Y}_i \\
g_i(\mathbf{z}_*)
\end{bmatrix}
\sim \mathcal{N} \left( \mathbf{0},
\begin{bmatrix}
\mathbf{K}_i(\mathbf{Z}, \mathbf{Z}) + \sigma_i^2 \mathbf{I}_M & \mathbf{k}_i(\mathbf{Z}, \mathbf{z}_*) \\
\mathbf{k}_i(\mathbf{z}_*, \mathbf{Z}) & k_i(\mathbf{z}_*, \mathbf{z}_*)
\end{bmatrix}
\right).
\end{equation}

The posterior predictive distribution yields
\begin{align}
\mu_{i,*} &= \mathbf{k}_i(\mathbf{z}_*, \mathbf{Z}) 
\left[ \mathbf{K}_i(\mathbf{Z}, \mathbf{Z}) + \sigma_i^2 \mathbf{I}_M \right]^{-1} \mathbf{Y}_i, \\
\sigma_{i,*}^2 &= k_i(\mathbf{z}_*, \mathbf{z}_*) - \mathbf{k}_i(\mathbf{z}_*, \mathbf{Z}) 
\left[ \mathbf{K}_i(\mathbf{Z}, \mathbf{Z}) + \sigma_i^2 \mathbf{I}_M \right]^{-1} 
\mathbf{k}_i(\mathbf{Z}, \mathbf{z}_*).
\end{align}

The hyperparameters $\boldsymbol{\theta}_i$ and noise variance $\sigma_i^2$ are estimated by minimizing the negative log marginal likelihood (NLML):
\begin{equation}
-\log P(\mathbf{Y}_i \mid \mathbf{Z}, \boldsymbol{\theta}_i) =
\frac{1}{2} \mathbf{Y}_i^\top \mathbf{\Sigma}_i^{-1} \mathbf{Y}_i
+ \frac{1}{2} \log |\mathbf{\Sigma}_i|
+ \frac{M}{2} \log 2\pi,
\end{equation}
where $\mathbf{\Sigma}_i = \mathbf{K}_i(\mathbf{Z}, \mathbf{Z}) + \sigma_i^2 \mathbf{I}_M$.

We employ a radial basis function kernel with automatic relevance determination (ARD):
\begin{equation}
k_i(\mathbf{z}_m, \mathbf{z}_l) =
(\theta_{i,1})^2 
\exp\left(
-\sum_{j=1}^{d+2} 
\frac{(z_{m,j} - z_{l,j})^2}{2(\theta_{i,j+1})^2}
\right),
\end{equation}
where the input dimensions correspond to the latent coordinates and the parameters $(Re,\alpha)$.

Finally, the GPR model is trained using time-shifted pairs
\begin{equation}
\bigl(\mathbf{y}_m(t),\,\mathbf{y}_m(t+T_p)\bigr),
\end{equation}
and the resulting surrogate for the latent dynamics is given by the discrete-time map
\begin{equation}
\left\{
\begin{aligned}
y_1(t+T_p) &= g_1(y_1(t), \dots, y_d(t), Re, \alpha), \\
&\vdots \\
y_d(t+T_p) &= g_d(y_1(t), \dots, y_d(t), Re, \alpha).
\end{aligned}
\right.
\label{eq:GPR_y}
\end{equation}

For the numerical evaluation of the surrogate models used to obtain the results reported in Section~\ref{sec:results}, only the expected (mean) value of the predicted solution operators in Eq.~\eqref{eq:GPR_y} is considered. The predictive accuracy and associated uncertainty levels are illustrated in Fig.~\ref{fig:uncertainty}. For input data $(\mathbf{y}^n_m, Re, \alpha)$ selected from the test set, the GPR model provides both the expected value of the prediction $\mathbf{y}^{n+1}_{m,\text{pred}}$ (black dots in Fig.~\ref{fig:uncertainty}(a)) and the associated uncertainty, quantified through the standard deviation and reported as 95\% confidence intervals (cyan bands). These predictions are compared against the corresponding true values $\mathbf{y}^{n+1}_{m,\text{true}}$ in panel (a), while the distribution of the prediction error, $\mathbf{y}^{n+1}_{m,\text{true}} - \mathbf{y}^{n+1}_{m,\text{pred}}$, over all test samples is shown in panel (b). The maximum standard deviation over the test set is $2.56\%$ of the corresponding mean value. Such relatively small uncertainty levels demonstrate that the learned surrogate models of the rotating cylinder flow provide an accurate and robust approximation of the latent Navier--Stokes dynamics across the explored parameter ranges.

\section{Sensitivity of the AE embedding to network weight initialization}
\label{app:UQ_AE}

\begin{figure}[t]
    \centering
    \includegraphics[scale=0.8]{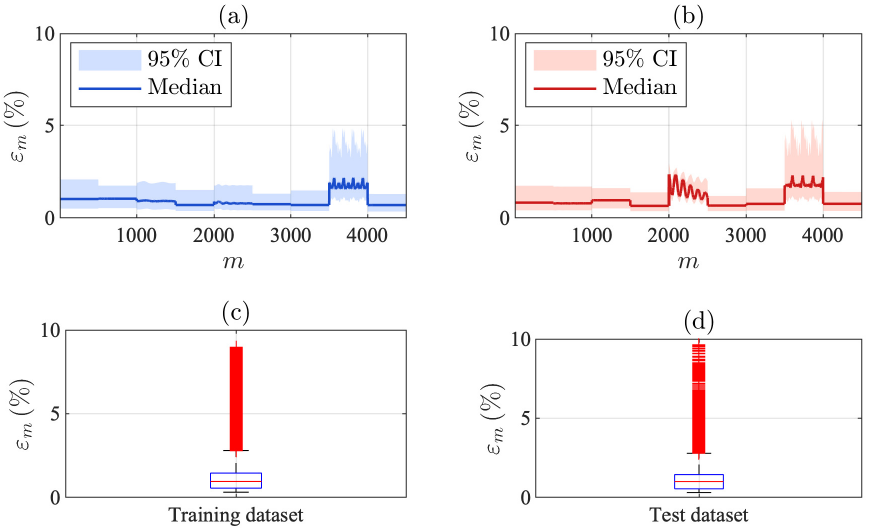}
    \caption{Sensitivity of the AE embedding to network weight initialization
    over $N_{\text{seed}}=100$ independent random initializations.
    Panels~(a)-(b): median (solid line) and 5--95\,\% confidence interval (shaded band)
    of the baseline reconstruction error $\varepsilon_m$ (\%) computed
    over the $m = 1,\dots,4500$ snapshots of the training and test datasets,
    respectively.
    Panels~(c)-(d): box-plot visualization of the same distributions shown in
    panels~(a)-(b), aggregating all snapshots and all $N_{\text{seed}}$ models into a single
    statistical summary for the training and test sets, respectively.
    Whiskers extend to the 5th and 95th percentiles; the central mark denotes
    the median. Note that, for visualization purposes, the snapshots in panels~(a)–(b) are ordered according to increasing \(Re\)–\(\alpha\) parameter combinations.
}
    \label{fig:uncertainty_AE}
\end{figure}

Convolutional autoencoder architectures embed high-dimensional flow fields into a
low-dimensional latent space through a sequence of nonlinear operations whose
learned parameters depend critically on the random initialization of the network
weights. As a result, different training runs starting from different weight
initializations may converge to distinct local minima of the loss landscape,
potentially yielding embeddings with non-negligible variability in reconstruction
accuracy. Assessing the sensitivity of the embedding to this source of randomness
is essential to establish the robustness of the coordinate system on which the
AE-based reduced-order models are built.

To this end, we select the best-performing autoencoder architecture identified in
Section~\ref{subsec:res_coord} --- namely AE$_4$ (see
Fig.~\ref{Manifold_baseline_recon_error_AUTOENCODERS} and
Table~\ref{tab:autoencoder_error_statistics}) --- and repeat the training procedure
for $N_{\text{seed}} = 100$ independent random initializations of all encoding and
decoding network weights, keeping all other hyperparameters (architecture, optimizer,
learning rate, batch size, early-stopping patience) identical across runs. In each run, all convolutional and dense layer weights are initialized by drawing
from the Glorot uniform distribution~\cite{glorot2010understanding}, with the
random seed set to $s \in \{0, 1, \dots, 99\}$ prior to model construction;
all other sources of randomness (data loading and train/test split) are
fixed across runs to isolate the effect of weight initialization alone.
This ensemble of $N_{\text{seed}}$ trained models allows us to evaluate the
sensitivity of the baseline reconstruction error (see
Eq.~\eqref{eq:reconstruction_error} in Appendix~\ref{app:NS_eq}) over the
$m = 1,\dots,4500$ snapshots of both the training dataset
($Re \in \{40, 50, 70\}$, $\alpha \in \{4.5, 5.5, 6.5\}$) and the test
dataset ($Re \in \{45, 55, 75\}$, $\alpha \in \{4.5, 5.5, 6.5\}$), retaining
only the last 500 snapshots per $(Re, \alpha)$ combination to focus on the
statistically converged regime.

Results are reported in Fig.~\ref{fig:uncertainty_AE}.
Panels~(a)--(b) display, for each snapshot index $m$, the median and 5--95\,\%
confidence interval of $\varepsilon_m$ across the $N_{\text{seed}}$ models for
the training and test datasets, respectively. Note that, for visualization purposes, the snapshots are ordered according to increasing \(Re\)–-\(\alpha\) parameter combinations. Consequently, lower values of \(m\) correspond to steady-state regimes, whereas higher values correspond to limit-cycle regimes.
The sensitivity to weight initialization is low throughout the steady-state regimes
($\varepsilon_m \lesssim 2\,\%$), confirming that the AE embedding is robust when
the flow settles on a fixed-point attractor.
The variability slightly increases in the limit-cycle (periodic) regimes,
reaching a maximum of approximately $\varepsilon_m \approx 5\,\%$ in the snapshot
intervals $m \in [2000, 2500]$ and $m \in [3500, 4000]$, which correspond to
parameter combinations where the cylinder wake exhibits sustained vortex shedding.
This suggests that, while the median reconstruction quality is consistently good,
the loss landscape becomes more complex in the periodic regimes, allowing different
initializations to settle into solutions with slightly different reconstruction
accuracy.

Crucially, the sensitivity is comparable between the training and test sets ---
both in the width and magnitude of the confidence bands --- indicating that the
observed variability is an intrinsic property of the optimization problem rather
than a symptom of overfitting.
Finally, panels~(c)--(d) provide a compact box-plot summary of the full
error distributions, aggregating all snapshots and all $N_{\text{seed}} = 100$
seeds into a single visualization for each dataset. The interquartile range and
whiskers confirm the picture emerging from panels~(a)--(b): the AE embedding
is robust across random initializations, with the bulk of reconstruction errors
well below $\varepsilon_m = 3\,\%$ for both training and test data.

\bibliographystyle{unsrt}  
\bibliography{mybibfile}
	
	\end{document}